\documentclass[a4paper,11pt, onesided]{article}
\usepackage[ansinew]{inputenc}

\usepackage{amsfonts}
\usepackage{amsmath}
\usepackage{amsthm}

\usepackage{mathrsfs}
\usepackage{amssymb}

\usepackage{enumerate}

\usepackage{amsbsy}
\usepackage{amscd}
\usepackage{epsf}
\usepackage{graphics}
\usepackage{graphicx}
\usepackage{psfrag}
\usepackage[final]{pdfpages}

\newcommand{\R}{\mathbb{R}}
\newcommand{\C}{\mathbb{C}}

\newtheorem{theorem}{Theorem}
\newtheorem{lemma}[theorem]{Lemma}
\newtheorem{proposition}[theorem]{Proposition}

\newtheorem{definition}{Definition}

\begin{document}

\title{Curves and envelopes that bound the spectrum of a matrix}
\author{G{\"o}ran Bergqvist\\
      Department of Mathematics, Link{\"o}ping University, \\
      SE-581 83 Link{\"o}ping, Sweden \\
      gober@mai.liu.se}
\date{}

\maketitle

\begin{abstract}

A generalization of the method developed by Adam, Psarrakos  and Tsatsomeros to find
inequalities for the eigenvalues of a complex matrix $A$ using knowledge of the largest 
eigenvalues of its Hermitian part $H(A)$ is presented. 
The numerical range or field of values of $A$ can be constructed as the intersection of half-planes
determined by the largest eigenvalue of $H(e^{i\theta}A)$. Adam, Psarrakos  and Tsatsomeros showed 
that using the two largest eigenvalues of $H(A)$, the eigenvalues of $A$ satisfy a cubic inequality and
the envelope of such cubic curves defines a region in the complex plane smaller than the 
numerical range but still containing the spectrum of $A$.
Here it is shown how using the three largest eigenvalues of $H(A)$ or more, one obtains new 
inequalities for the eigenvalues of $A$ and new envelope-type regions containing the spectrum of $A$. 
\end{abstract}


\vskip.2cm
\noindent
AMS classification codes: 15A18, 15A42, 15A60, 65F15 

%----------------------------------------------------------------------------------------------------------------------

\section{Introduction}

In this paper, we denote by ${\cal  M}_{n,k}(\C)$ and ${\cal  M}_{n,k}(\R)$ the spaces of complex 
and real $n\times k$ matrices
respectively;  ${\cal  M}_{n}(\C)$ and ${\cal  M}_{n}(\R)$ stand for $k=n$.
The spectrum $\sigma(A)$
%=\{\lambda\in\C ; \ \exists\bold  x\in\C^n, \bold x\ne \bold 0 , A\bold x=\lambda\bold x\}$
of a matrix $A\in {\cal  M}_{n}(\C)$ is known to be located in its numerical range or field of values 
$F(A)=\{\bold x^*A\bold x \in \C ;\ \bold x\in\C^n, ||\bold x||_2=1\}$. 
%This follows immediately from
%$\sigma(A)\ni\lambda =\bold x^*A\bold x$, if $\bold x$ is chosen as a unit length eigenvector.
The spectrum of $A$ is also located to the left of the vertical line $Re(z)=\delta_1$ in the complex plane, where
$\delta_1$ is the largest eigenvalue of the Hermitian part $H(A)=\frac{1}{2}(A+A^*)$ of $A$. Here 
$A^*$ denotes the Hermitian conjugate of $A$. Also, $S(A)=\frac{1}{2}(A-A^*)$ 
denotes the skew-Hermitian part of $A$, so $A=H(A)+S(A)$.
%Again this follows easily from $\lambda +\bar\lambda=\bold x^*A\bold x+(\bold x^*A\bold x)^*=
%2\bold x^*H(A)\bold x\le 2\delta_1\bold x^*\bold x=2\delta_1$ if $\bold x$ is  a unit eigenvector.

\smallskip

Clearly $F(A)=e^{-i\theta}F(e^{i\theta}A)$ for any $\theta\in [0,2\pi[$, so 
$e^{i\theta}F(A)$ is located to the left of the vertical line $Re(z)=\lambda_{max}(H(e^{i\theta}A))$, and
rotating this line by $e^{-i\theta}$ we get a new line that bounds $\sigma(A)$. In fact, \cite{HJ2,J},
$F(A)$ is obtained exactly as the connected, compact and convex region defined by the envelope of 
all such lines. In the first subfigure of Figure \ref{fig1} we have illustrated this for the Toeplitz matrix
\begin{equation}\label{eq:Toep}
A=\begin{pmatrix} 1 & 1 & 0 & i \\  2 & 1 & 1 & 0 \\  3 & 2 & 1 & 1 \\4 & 3 & 2 & 1 \\\end{pmatrix}\ ,
\end{equation}
by plotting these lines for $\theta= 2\pi m/120, m=0, \dots ,119$. 
%The other subfigures illustrate the region $\mathcal{E}_1(A)$ as investigated in  \cite{PT1} and
%the region $\mathcal{E}_2(A)$ obtained by using 3 eigenvalues of $H(A)$ of the Toeplitz matrix 
% $A$ in (\ref{eq:Toep}).
 The eigenvalues of $A$ are marked by small boxes in the figure.

In \cite{AT} Adam and Tsatsomeros showed how one can use the two largest eigenvalues
$\delta_1$ and $\delta_2$ of $H(A)$ and the eigenvector $\bold u_1$ of $H(A)$ corresponding
to $\delta_1$, to obtain an improved inequality for $\sigma(A)$. Let 
$\alpha=Im(\bold u_1^*S(A)\bold u_1)$ and $K_1=||S(A)\bold u_1||_2^2-\alpha^2\ge 0$. Then
they proved that any $\lambda\in\sigma(A)$ satisfies
\begin{equation}\label{eq:cubic}
|\lambda-(\delta_1+i\alpha)|^2(Re(\lambda)-\delta_2)\le K_1(\delta_1-Re(\lambda))\ .
\end{equation}
With equality in (\ref{eq:cubic}) we have a curve $\Gamma_1(A)$ that bounds $\sigma(A)$ and is of 
degree 3 in the real and imaginary parts of $\lambda$. 
Applied to $H(e^{i\theta}A)$ one can repeat the argument above and obtain
rotated cubic curves that bound $\sigma(A)$. The envelope of such curves was studied
extensively by Psarrakos and Tsatsomeros \cite{PT1,PT2} and they showed that it bounds a region
$\mathcal{E}_1(A)$ that contains $\sigma(A)$ and is compact, but which is not always convex or connected.
In the second subfigure of Figure \ref{fig1} we show their image \cite{PT1} 
of $\mathcal{E}_1(A)$ for the Toeplitz matrix $A$ in (\ref{eq:Toep}).
Again we used 120 curves, with $\theta= 2\pi m/120, m=0, \dots ,119$, for the plot.

\medskip

The aim of this paper is to generalize the results of Adam, Psarrakos and Tsatsomeros by
using the $k$ largest eigenvalues of $H(A)$ to obtain new curves $\Gamma_k(A)$ that bound $\sigma(A)$. 
As a preview of our results, we show in the third subfigure of Figure \ref{fig1} the region $\mathcal{E}_2(A)$ 
obtained from the envelope of curves when the three largest eigenvalues of $H(A)$ are utilized for the 
Toeplitz matrix $A$ in (\ref{eq:Toep}). Also here 120 curves are used to construct the figure.

\begin{figure}
\includepdf[pages=1,pagecommand={},offset=-55mm 80mm,
trim = 25mm 120mm 75mm 20mm,clip,width=55mm]{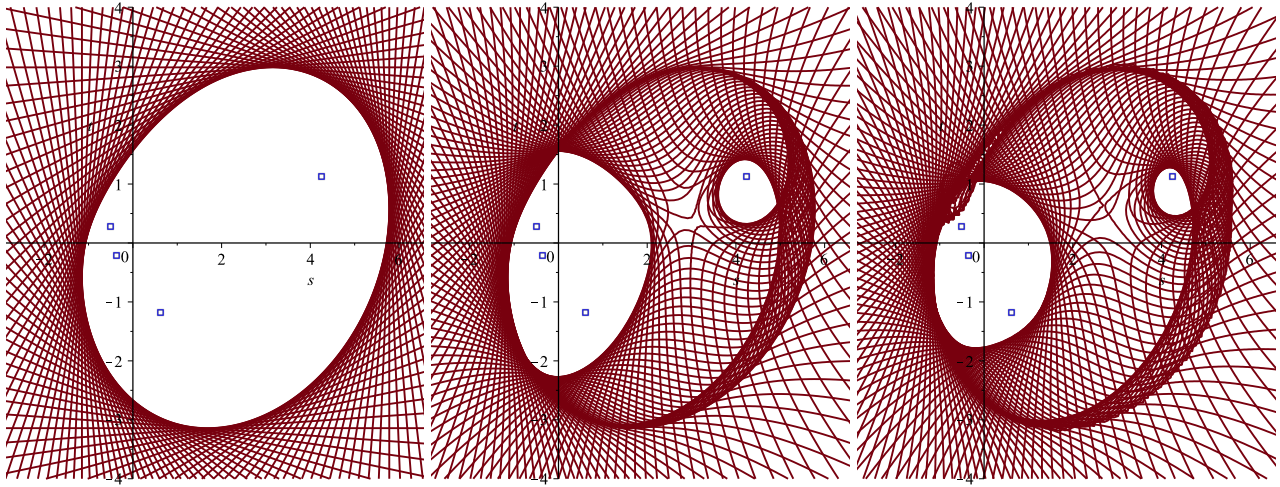}
\includepdf[pages=2,pagecommand={},offset=1mm 80mm,
trim = 25mm 120mm 75mm 20mm,clip,width=55mm]{EVFig.pdf}
\includepdf[pages=3,pagecommand={},offset=57mm 80mm,
trim = 25mm 120mm 75mm 20mm,clip,width=55mm]{EVFig.pdf}
\vskip52mm
\caption{\label{fig1} The numerical range $F(A)$ (left), the regions $\mathcal{E}_1(A)$ (middle)  and
 $\mathcal{E}_2(A)$ (right)  of $A$ in (\ref{eq:Toep}).}
\end{figure}

\medskip

In Section 2 we derive the main inequality for the eigenvalues of a matrix, with respect to
 the largest eigenvalues and corresponding eigenvectors of its Hermitian part.
We also state a more explicit formulation for the case of three known eigenvalues,
and present some illustrations and analyze properties of the curves that
bound the spectrum. In Section 3 we analyze explicitly some cases where the
curves have special properties. Then, in Section 4, we demonstrate how an envelope
of such curves encloses a region which is inside the numerical range and contains the spectrum.
We compare this region with the region obtained for the case of two known eigenvalues
presented by Psarrakos and Tsatsomeros.
 
 %----------------------------------------------------------------------------------------------------------------------

\section{Eigenvalue inequalities and spectrum bounding curves}

Let $A\in {\cal  M}_n(\C)$ and let  $\delta_1\ge \dots \ge \delta_n$ be the eigenvalues of 
$H(A)$ with $\bold{u}_1,\dots,\bold{u}_n$ the corresponding normalized orthogonal eigenvectors. 
Consider the unitary matrix $U\in {\cal  M}_n(\C)$ with columns $\bold{u}_j$, the diagonal matrix
 $\Delta=diag(\delta_{1}, \dots ,\delta_{n})\in {\cal  M}_n(\R)$ with diagonal elements  
$\delta_{1}, \dots ,\delta_{n}$, and formulate the Hermitian part of $A$ as
\begin{equation}\label{eq:21}
H(A)=U\Delta U^* \Leftrightarrow U^*H(A)U=\Delta=
\begin{pmatrix} \Delta_k  &  0 \\ 0 &  \tilde\Delta_k \end{pmatrix},
\end{equation}
where $\Delta_k=diag(\delta_{1}, \dots ,\delta_{k})\in {\cal  M}_{k}(\R)$ and 
$\tilde\Delta_k=diag(\delta_{k+1}, \dots ,\delta_{n})\in {\cal  M}_{n-k}(\R)$. Define the skew-Hermitian matrix
\begin{equation}\label{eq:22}
Y=U^*S(A)U=\begin{pmatrix} Y_k &  -V_k^*  \\  V_k& \tilde Y_k  \end{pmatrix} ,
\end{equation}
where $Y_k\in {\cal  M}_k(\C)$ and $\tilde Y_k\in {\cal  M}_{n-k}(\C)$ are skew-Hermitian, and 
$V_k\in {\cal  M}_{n-k,k}(\C)$.
Now, let the upper principal $k\times k$ submatrix of $U^*(A-\lambda I_n)U$ be 
\begin{equation}\label{eq:23}
W_k=\Delta_k+Y_k-\lambda I_k \ ,
\end{equation} 
where $\lambda=s+it$ is an eigenvalue of $A$. Obviously, combining (\ref{eq:21}) - (\ref{eq:23})
for the Hermitian part of $W_k$, we have $H(W_k)=\frac{1}{2}(W_k+W_k^*)=\Delta_k-sI_k\,$. 
\begin{lemma}
If $\lambda\in\sigma(A)\smallsetminus\sigma(\Delta_k+Y_k)$, then $H(W_k^{-1})$ is not negative definite.
\end{lemma}
\begin{proof}
$W_k^{-1}$ exists since $\lambda\notin\sigma(\Delta_k+Y_k)$. Since $s\le\delta_1$ for $\lambda\in\sigma(A)$,  $H(W_k)=diag(\delta_1-s, \dots , \delta_k-s)$ has at least one
non-negative eigenvalue $\delta_1-s$ and is therefore not negative definite, which is equivalent to
$H(W_k^{-1})$  not being negative definite  \cite{M}.
\end{proof}

Recall that the adjugate $adj(A)\in {\cal  M}_{n}(\C)$ of $A\in {\cal  M}_{n}(\C)$ is the matrix
whose elements are minors of $A$, and satisfies $adj(A)A=(\det A) I_n$.
The following is the main result of the paper.

\begin{theorem}\label{NH}
Let $A\in {\cal  M}_n(\C)$ and $\lambda $ be an eigenvalue of $A$. Let
   $\delta_1\ge \dots \ge \delta_n$ be the eigenvalues of  the
   Hermitian part of $A$. Then
\begin{equation}\label{eq:main}
|\det W_k|^2 (Re(\lambda)-\delta_{k+1}) \le (\sigma_1(V_k))^2
\lambda_{max}(H(\det W_k adj(W_k^*)))\ ,
\end{equation}
where  $V_k$ and  $W_k$ are defined in (\ref{eq:22}) and (\ref{eq:23}) respectively, 
$adj(W_k^*)$ is the adjugate of $W_k^*$, $\lambda_{max}(H(\det W_k adj(W_k^*)))\ge 0$ 
is the largest eigenvalue of the Hermitian part of  $\det W_k adj(W_k^*)$, and 
$\sigma_1(V_k)$ is the largest singular value of $V_k$. 
\end{theorem}

\begin{proof}

We adopt the ideas from the proof of the case $k=1$ presented by Adam and Tsatsomeros \cite{AT}.
Let $\lambda = s+it$ be an eigenvalue of $A$. This means that $A-\lambda I_n$ and
\begin{equation*}
U^*(A-\lambda I_n)U=\begin{pmatrix} W_k &  -V_k^*  \\  V_k & \tilde\Delta_k+\tilde Y_k-\lambda I_{n-k}  \end{pmatrix} ,
\end{equation*}
are singular. If $W_k$ is singular, then $\det W_k=0$ and the statement of the theorem is trivial.
For the case that $W_k$ is nonsingular, the Schur complement  
$\tilde W_k=\tilde\Delta_k+\tilde Y_k-\lambda I_{n-k}+V_k W_k^{-1}V_k^*$ of $W_k$ 
must be singular, since $0=\det (U^*(A-\lambda I_n)U)=\det W_k\det \tilde W_k$.
Hence $0\in\sigma(\tilde W_k)\subset F(\tilde W_k)$, which implies 
$0\in Re( F(\tilde W_k))=F(H(\tilde W_k))$, \cite{HJ2}. Consequently there exists a unit vector
$\bold x\in \C^{n-k}$ such that
\begin{equation}\label{eq:xHx}
\bold x^* H(\tilde W_k)\bold x =0\ .
\end{equation}
Since
\begin{equation*}
H(\tilde W_k)=\frac{1}{2}(\tilde W_k+\tilde W_k^*)=\tilde\Delta_k-sI_{n-k}+\frac{1}{2}V_k(W_k^{-1}+
(W_k^{-1})^*)V_k^*=\tilde\Delta_k-sI_{n-k}+V_kH(W_k^{-1})V_k^* \, ,
\end{equation*}
from (\ref{eq:xHx}) we get  
\begin{equation}\label{eq:xdx}
0=\bold x^*\tilde\Delta_k \bold x-s \bold x^* I_{n-k}\bold x+
\bold x^*V_kH(W_k^{-1})V_k^*\bold x=\bold x^*\tilde\Delta_k \bold x-s +
(V_k^*\bold x )^*H(W_k^{-1})(V_k^*\bold x)\ .
\end{equation}
Further, using the unit vector $\bold x$,  the largest eigenvalues $\delta_{k+1}$ and 
$\lambda_{max}(H(W_k^{-1}))$ of the Hermitian matrices $\tilde\Delta_k$ and
$H(W_k^{-1})$, respectively, and the largest singular value $\sigma_1(V_k^*)$ of $V_k^*$, we obtain 
\begin{equation}\label{eq:xd2x}
\bold x^*\tilde\Delta_k \bold x\le \delta_{k+1}||\bold x||_2^2=\delta_{k+1}\ ,
\end{equation}
\begin{equation}\label{eq:xh2x}
(V_k^*\bold x )^*H(W_k^{-1})(V_k^*\bold x) \le \lambda_{max}(H(W_k^{-1}))||V_k^*\bold x||_2^2 \ ,
\end{equation}
and
\begin{equation}\label{eq:Vx}
||V_k^*\bold x||_2^2\le (\sigma_1(V_k^*))^2 ||\bold x||_2^2= (\sigma_1(V_k^*))^2=(\sigma_1(V_k))^2 \ .
\end{equation}
Also, according to Lemma 1 for the largest eigenvalue of $H(W_k^{-1})$ we have $\lambda_{max}(H(W_k^{-1}))\ge 0$,
which together with (\ref{eq:Vx}) gives
\begin{equation}\label{eq:xh3x}
\lambda_{max}(H(W_k^{-1}))||V_k^*\bold x||_2^2\le \lambda_{max}(H(W_k^{-1}))(\sigma_1(V_k))^2 \ .
\end{equation}
Using  the inequalities (\ref{eq:xd2x}), (\ref{eq:xh2x})  and (\ref{eq:xh3x}) in (\ref{eq:xdx}) we get  
\begin{equation}\label{eq:A}
0\le \delta_{k+1}-s +\lambda_{max}(H(W_k^{-1}))(\sigma_1(V_k))^2   \ .
\end{equation}
The inequality (\ref{eq:A}) is our main result in the case of a nonsingular $W_k$. In order to formulate 
(\ref{eq:main}), which is valid also for a singular $W_k$, notice that for $H(W_k^{-1})$ we have
\begin{equation*}
H(W_k^{-1})=\frac{1}{2} ( W_k^{-1}+(W_k^{-1})^*)=\frac{1}{2}(\frac{1}{\det W_k}adj(W_k)+
\frac{1}{{\overline{\det W_k} }} adj(W_k^*))=\frac{M_k}{|\det W_k|^2}\ ,
\end{equation*}
where
\begin{equation}\label{eq:M}
M_k=\frac{1}{2}( {\overline{\det W_k} }adj(W_k)+\det W_k adj(W_k^*))=H(\det W_k adj(W_k^*))\ .
\end{equation}
Hence, 
\begin{equation}\label{eq:A2}
\lambda_{max}(H(W_k^{-1}))= \frac{1}{|\det W_k|^2}  \lambda_{max}(M_k) \ .
\end{equation}
Multiplying (\ref{eq:A})  by $|\det W_k|^2$ and using (\ref{eq:A2}), the inequality (\ref{eq:main}) follows.

\end{proof}

Note that $\lambda_{max}(M_k)$ is a non-negative function of $s$ and $t$ while $\sigma_1(V_k)$
is a constant.

\medskip

For $k=1$, $Y_1=i\alpha$ where $\alpha=-i\bold u_1^*S(A)\bold u_1 \in\R$, $V_1$ is a vector, and 
$W_1=\delta_1-s+i(\alpha-t)$ is a scalar. Interpreting the adjugate of 
an $1\times 1$ matrix as 1 (to keep $adj(A)A=(\det A) I_n$ that was used in the proof),
using (\ref{eq:23}) and (\ref{eq:M}) we have $M_1=\delta_1-s$ and 
$\lambda_{max}(M_1)=M_1$. 
Moreover, for the column $\bold{y}_1$ of $Y$ in (\ref{eq:22}) we have
 $||S(A)\bold u_1||_2^2=\bold{y}_1^*\bold{y}_1=\alpha^2+V_1^*V_1$ $\Rightarrow$
 $K_1=(\sigma_1(V_1))^2\equiv V_1^*V_1= ||S(A)\bold u_1||_2^2-\alpha^2 $, whereby we obtain 
 the result of Adam and Tsatsomeros  \cite{AT}   
\begin{equation}\label{eq:cub}
[(\delta_1-s)^2+(\alpha-t)^2](s-\delta_2) \le K_1(\delta_1-s)\ .
\end{equation}

\noindent
We use the notation $\Gamma_1(A)$ for the cubic curve obtained from
equality in (\ref{eq:cub}). More generally:
\begin{definition}
$\Gamma_k(A)$ is the curve $|\det W_k|^2 (Re(\lambda)-\delta_{k+1}) = (\sigma_1(V_k))^2
\lambda_{max}(H(\det W_k adj(W_k^*)))$ 
obtained from having equality in  (\ref{eq:main}) of Theorem \ref{NH}.
\end{definition}
\noindent
We use the variables $s$ and $t$ for the curve, where $s+it=\lambda$ in the expression (\ref{eq:23})
for $W_k$.

\smallskip
For $k=2$ we now state an explicit inequality which gives an expression for the curve $\Gamma_2(A)$.
In this case
\begin{equation}\label{eq:W2}
W_2=\Delta_2+Y_2-\lambda I_2=\begin{pmatrix} \delta_1-s+i(\alpha-t) &  -\bar\gamma   \\  
\gamma & \delta_2-s+i(\beta-t)  \end{pmatrix} ,
\end{equation}
where $i\alpha=\bold u_1^*S(A)\bold u_1$, $i\beta=\bold u_2^*S(A)\bold u_2$, 
$\gamma=\bold u_2^*S(A)\bold u_1$, and $V_2\in {\cal  M}_{n-2,2}(\C)$ is such that 
the first two columns of $Y$ in (\ref{eq:22}) are
\begin{equation}\label{eq:V}
{\cal  M}_{n,2}(\C)\ni\begin{pmatrix} Y_2   \\  V_2  \end{pmatrix} =
\begin{pmatrix} i\alpha & -\bar\gamma   \\  \gamma & i\beta \\ \bold v_1 & \bold v_2  \end{pmatrix}=
U^*S(A)\begin{pmatrix} \bold u_1 & \bold u_2  \end{pmatrix}\ .
\end{equation}

\begin{proposition}\label{3H}
   Let $A\in {\cal  M}_n(\C)$ and  $\lambda =s+it$ be an eigenvalue of $A$. Let
   $\delta_1\ge \delta_2\ge \delta_3$ be the three largest eigenvalues of the
   Hermitian part of $A$, and let $\bold u_1$ and $\bold u_2$ be the corresponding normalized
   eigenvectors of $\delta_1$ and $ \delta_2$. Then
\begin{equation*}
[((\delta_1-s)(\delta_2-s)-(\alpha-t)(\beta-t)+|\gamma|^2)^2+((\delta_1-s)(\beta-t)+(\delta_2-s)(\alpha-t))^2]
(s-\delta_3) 
\end{equation*}
\begin{equation}\label{eq:Th3}
\le\  \frac{K_2}{2}[m_1(s,t)+m_3(s,t)+
\sqrt{[m_1(s,t)-m_3(s,t)]^2+4|m_2(s,t)|^2}\,]\ ,
\end{equation}
where
\begin{equation*}
\alpha=-i\bold u_1^*S(A)\bold u_1 \in\R\ ,  \qquad  \beta=-i\bold u_2^*S(A)\bold u_2 \in\R\ , \qquad
\gamma=\bold u_2^*S(A)\bold u_1 \in\C \, \qquad \qquad 
\end{equation*}
\vspace{-7mm}
\begin{eqnarray*}
K_2=\frac{1}{2}[||S(A){\bold u}_1||_2^2+||S(A){\bold u}_2||_2^2-
\alpha^2-\beta^2-2|\gamma |^2 \hskip59mm \\
\qquad\quad  +\sqrt{(||S(A){\bold u}_1||_2^2-||S(A){\bold u}_2||_2^2-\alpha^2+\beta^2)^2+
4|(S(A){\bold u}_2)^*(S(A){\bold u}_1)+ i\gamma(\alpha+\beta)|^2}\ ] \ ,
\end{eqnarray*}
\begin{equation*}
m_1(s,t)=(\delta_1-s)[(\delta_2-s)^2+(\beta-t)^2]+(\delta_2-s)|\gamma|^2\ , \hskip49mm 
\end{equation*}
\begin{equation*}
m_2(s,t)=i\gamma[(\delta_1-s)(\beta-t)+(\delta_2-s)(\alpha-t)]\ , \hskip61mm 
\end{equation*}
\begin{equation*}
m_3(s,t)=(\delta_2-s)[(\delta_1-s)^2+(\alpha-t)^2]+(\delta_1-s)|\gamma|^2\ . \hskip49mm 
\end{equation*}
\end{proposition}
\vskip.2cm

\begin{proof}
By (\ref{eq:W2}), we have
\begin{equation*}
\det W_2=(\delta_1-s)(\delta_2-s)-(\alpha-t)(\beta-t)+|\gamma|^2+i((\delta_1-s)(\beta-t)+(\delta_2-s)(\alpha-t))\ ,
\end{equation*}
which gives
\begin{equation}\label{eq:B1}
|\det W_2|^2=((\delta_1-s)(\delta_2-s)-(\alpha-t)(\beta-t)+|\gamma|^2)^2+((\delta_1-s)(\beta-t)+
(\delta_2-s)(\alpha-t))^2 \ .
\end{equation}
Furthermore
\begin{equation*}
adj (W_2^*)=adj \begin{pmatrix} \delta_1-s-i(\alpha-t) &  \bar\gamma   \\  
-\gamma & \delta_2-s-i(\beta-t)  \end{pmatrix} =
\begin{pmatrix} \delta_2-s-i(\beta-t) &  -\bar\gamma   \\  \gamma & \delta_1-s-i(\alpha-t)  \end{pmatrix} ,
\end{equation*}
and a straightforward calculation from (\ref{eq:M}) gives
\begin{eqnarray}\label{eq:B2}
M_2=\begin{pmatrix} m_1(s,t) &  \bar m_2(s,t)   \\  m_2(s,t) & m_3(s,t)  \end{pmatrix} =
H(\det W_2 adj(W_2^*))=\frac{1}{2}( {\overline{\det W_2} }adj(W_2)+\det W_2 adj(W_2^*))= \nonumber \\
\begin{pmatrix} (\delta_1-s)[(\delta_2-s)^2+(\beta-t)^2]+(\delta_2-s)|\gamma|^2 &  
-i\bar\gamma[(\delta_1-s)(\beta-t)+(\delta_2-s)(\alpha-t)]   \\  
i\gamma[(\delta_1-s)(\beta-t)+(\delta_2-s)(\alpha-t)] & 
(\delta_2-s)[(\delta_1-s)^2+(\alpha-t)^2]+(\delta_1-s)|\gamma|^2 \end{pmatrix} \, .\ 
\end{eqnarray}
The largest eigenvalue of $M_2$ is
\begin{equation}\label{eq:B3}
\lambda_{max}(M_2)=\frac{1}{2}(m_1+m_3+\sqrt{(m_1-m_3)^2+4|m_2|^2})\ .
\end{equation}

To calculate $K_2=(\sigma_1(V_2))^2=\lambda_{max}(V_2^*V_2)$,
use that  $V_2=({\bold v}_1\ \,{\bold v}_2)$ by (\ref{eq:V}). Then
$V_2^*V_2$ is the $2\times 2$ Gram matrix
\begin{equation*}
V_2^*V_2=\begin{pmatrix} {\bold v}_1^*{\bold v}_1 &  {\bold v}_1^*{\bold v}_2   \\  
{\bold v}_2^*{\bold v}_1 & {\bold v}_2^*{\bold v}_2  \end{pmatrix}\ ,
\end{equation*}
whose largest eigenvalue is 
\begin{equation}\label{eq:K2}
K_2=\frac{1}{2}[{\bold v}_1^*{\bold v}_1+{\bold v}_2^*{\bold v}_2+
\sqrt{({\bold v}_1^*{\bold v}_1-{\bold v}_2^*{\bold v}_2)^2+4|{\bold v}_2^*{\bold v}_1|^2}\ ]\ .
\end{equation}
Further, for the columns $\bold y_j$  of $Y$ in (\ref{eq:22}) we have 
\begin{equation*}
\bold y_j^*\bold y_k=
(US(A){\bold u}_j)^*(US(A){\bold u}_k)=(S(A){\bold u}_j)^*U^*U(S(A){\bold u}_k)=
(S(A){\bold u}_j)^*(S(A){\bold u}_k) \ .
\end{equation*}
Thus, combining (\ref{eq:V}) with the above equalities we get
\begin{equation*}
\alpha^2+|\gamma |^2+{\bold v}_1^*{\bold v}_1=\bold y_1^*\bold y_1=||S(A){\bold u}_1||_2^2 \  ,\ 
|\gamma |^2+\beta^2+{\bold v}_2^*{\bold v}_2=\bold y_2^*\bold y_2=||S(A){\bold u}_2||_2^2\ ,
\end{equation*}
and 
\begin{equation*}
- i\gamma(\alpha+\beta)+{\bold v}_2^*{\bold v}_1=
\bold y_2^*\bold y_1=(S(A){\bold u}_2)^*(S(A){\bold u}_1)\ .
\end{equation*}
The substitution of these relations in (\ref{eq:K2}) yields
\begin{eqnarray}\label{eq:B4}
K_2=\frac{1}{2}[||S(A){\bold u}_1||_2^2+||S(A){\bold u}_2||_2^2-
\alpha^2-\beta^2-2|\gamma |^2 \hskip57mm \nonumber \\
\qquad +\sqrt{(||S(A){\bold u}_1||_2^2-||S(A){\bold u}_2||_2^2-\alpha^2+\beta^2)^2+
4|(S(A){\bold u}_2)^*(S(A){\bold u}_1)+ i\gamma(\alpha+\beta)|^2}\ ]\ .
\end{eqnarray}

Combining the equations (\ref{eq:B1}) , (\ref{eq:B2}) , (\ref{eq:B3}), (\ref{eq:B4}) and
(\ref{eq:main}), the inequality in (\ref{eq:Th3}) is derived for $k=2$.

\end{proof}

%----------------

We now state some properties of the curves $\Gamma_k(A)$ of Definition 1,
which for $k=1$ have been proved in \cite{PT1}.

\begin{proposition}\label{Gk}
For the curves $\Gamma_k(A)$ we have: 

\smallskip
(i) $\ \ \, \, \ \Gamma_k(\tilde U^* A\tilde U)=\Gamma_k(A)$ if $\tilde U$ is unitary

\smallskip
(ii) $\ \ \ \Gamma_k(A^T)=\Gamma_k(A)$

\smallskip
(iii) $\ \ \, \Gamma_k(A^*)=\overline{\Gamma_k(A)}$

\smallskip
(iv) $\ \ \ \Gamma_k(rA+bI_n)=r\Gamma_k(A)+b\ $ if $\ 0<r\in\R$ and $b\in\C$

\end{proposition}
\vskip.2cm

\begin{proof}
%Let $\tilde A=\tilde U^* A\tilde U$. Using the notation of (\ref{eq:21}), (\ref{eq:22}), (\ref{eq:23}) 
%and Theorem \ref{NH}, $\Delta+Y=U^*AU=(\tilde U^* U)^*\tilde A(\tilde U^* U)$ so $\Delta$
%and $Y$ are invariant, and since $\sigma(A)=\sigma(\tilde A)$ the first statement follows.
$(i)$ With $\tilde A = \tilde U^* A \tilde U$ we get $H( \tilde A) = \tilde U^* H(A) \tilde U$, 
$S( \tilde A) = \tilde U^* S(A) \tilde U$. Using the notation of (\ref{eq:21}), (\ref{eq:22}), 
(\ref{eq:23}) and Theorem \ref{NH}, we have that $H( \tilde A)$ is diagonalized by $\tilde U^* U\,$: 
$(\tilde U^* U)^*H( \tilde A)(\tilde U^* U)=U^*H(A)U=\Delta$, so $\Delta$ is invariant. Then
$(\tilde U^* U)^*S( \tilde A)(\tilde U^* U)=U^*S(A)U=Y$ is also invariant, which implies that all quantities 
$W_k$, $V_k$, and $\delta_{k+1}$ that are used in Theorem \ref{NH} are unchanged by a unitary
similarity transformation on $A$, hence $(i)$ follows.

$(ii)$ $\sigma(A^T)=\sigma(A)$ and $A^T$ will transform $W_k$ and $V_k$ into $W_k^T$ and $-V_k^*$,
respectively, which leave Theorem \ref{NH} invariant.

$(iii)$ $\sigma(A^*)=\overline{\sigma(A)}$ and $A^*$ will transform  $V_k$ into  $-V_k$ 
but $A^*-\lambda I_n$ gives a $W_k(A^*;\lambda)=\Delta_k-Y_k-\lambda I_k=(\Delta_k+Y_k-\bar\lambda I_k)^*=
(W_k(A;\bar\lambda))^*$ so Theorem \ref{NH} becomes a statement for $\bar\lambda$ if
$\lambda\in\sigma(A^*)$. 

$(iv)$ $A+bI_n$ leaves $V_k$ unchanged and replaces $-\lambda I_k$ by
$(b-\lambda) I_k$ in $W_k$. For $r>0,\ \Gamma_k(rA)=\{(s+it)/r\in\Gamma_k(A)\}=r\Gamma_k(A)$
because both sides of (\ref{eq:main}) in Theorem \ref{NH} then scale as $r^{2k+1}$.
\end{proof}

\smallskip

Next we give some illustrations of Proposition \ref{3H}. In Figure \ref{fig2} we depict the curve 
$\Gamma_2(A)$ defined by having equality in (\ref{eq:Th3}) for three random 
$5\times 5$ complex matrices with elements that have 
real and imaginary parts between -1 and 1, and compare it with the cubic curve 
$\Gamma_1(A)$ defined by having equality in (\ref{eq:cub})  \cite{AT}. 
The eigenvalues  are marked by small boxes. It is usually the case that the new curve is a strict improvement, 
i.e., all points in $\C$ satisfying the inequality for $k=2$ also satisfy the inequality  for $k=1$. 
However, we shall see below that this is not always the case.

\begin{figure}[!ht]
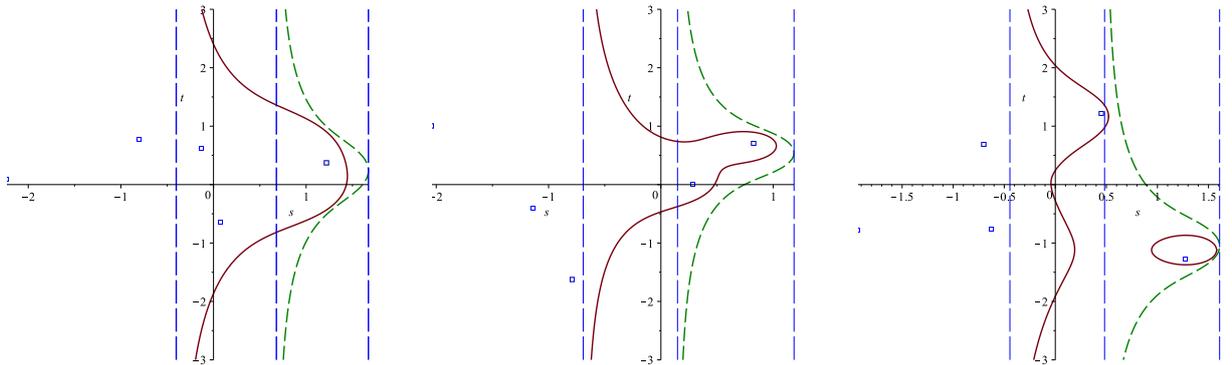

\includepdf[pages=4,pagecommand={},offset=-55mm -61mm,
trim = 25mm 122mm 35mm 22mm,clip,width=55mm]{EVFig.pdf}
\includepdf[pages=5,pagecommand={},offset=1mm -61mm,
trim = 25mm 122mm 35mm 22mm,clip,width=55mm]{EVFig.pdf}
\includepdf[pages=6,pagecommand={},offset=57mm -61mm,
trim = 25mm 122mm 35mm 22mm,clip,width=55mm]{EVFig.pdf}
\vskip35mm
\caption{\label{fig2} The curves $\Gamma_2(A)$ (solid) and $\Gamma_1(A)$ (dashed)
 for three complex $5\times 5$ matrices. The vertical lines are $s=\delta_3$,
$s=\delta_2$ and $s=\delta_1$. }
\end{figure}

In Figure \ref{fig3} we have used six real random $5\times 5$ matrices with elements between -1 and 1 
to show some more possible configurations. From simulations it is clear that for random matrices
the plots of Figure \ref{fig3} appear with decreasing probability. The first plots, with less interesting topology,
appear  more frequently but we shall see how one can construct matrices with all types
of curves that are seen in the  figure. 
In these figures the $k=2$ case is a strict improvement on the $k=1$ case.

\begin{figure}
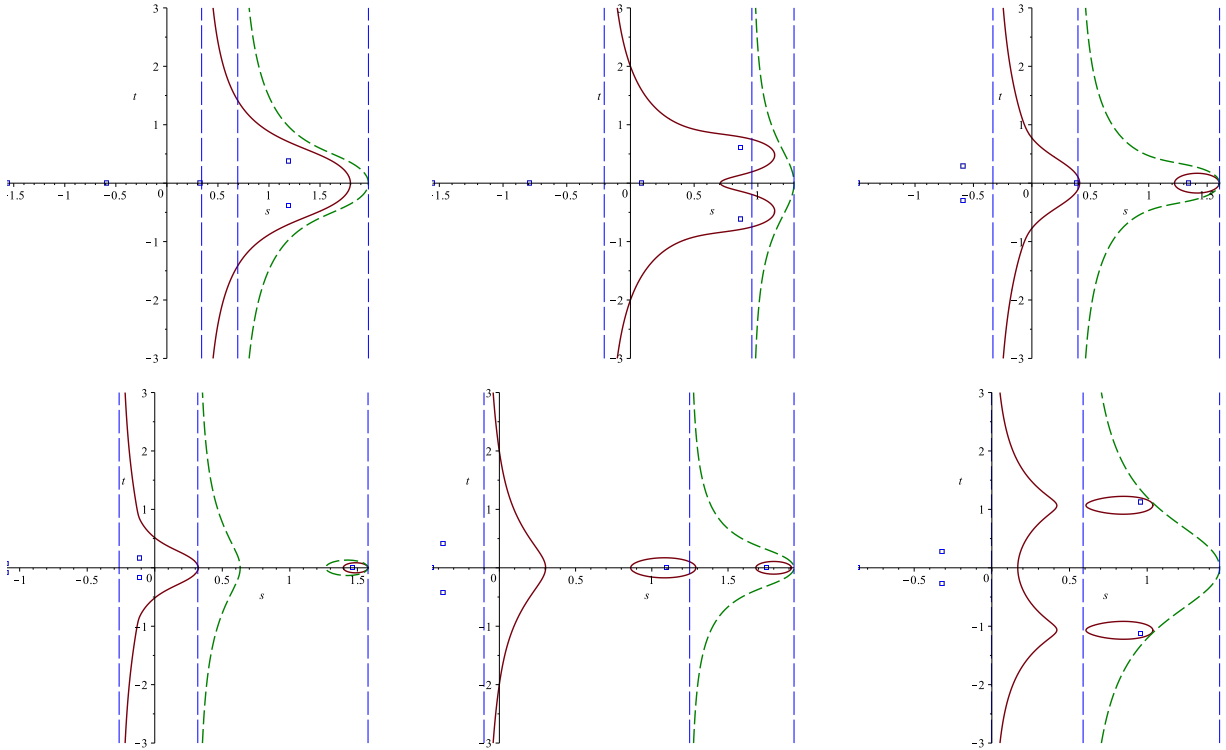

\includepdf[pages=7,pagecommand={},offset=-55mm 80mm,
trim = 25mm 122mm 35mm 22mm,clip,width=55mm]{EVFig.pdf}
\includepdf[pages=8,pagecommand={},offset=1mm 80mm,
trim = 25mm 122mm 35mm 22mm,clip,width=55mm]{EVFig.pdf}
\includepdf[pages=9,pagecommand={},offset=57mm 80mm,
trim = 25mm 122mm 35mm 22mm,clip,width=55mm]{EVFig.pdf}
\includepdf[pages=10,pagecommand={},offset=-55mm 29mm,
trim = 25mm 122mm 35mm 22mm,clip,width=55mm]{EVFig.pdf}
\includepdf[pages=11,pagecommand={},offset=1mm 29mm,
trim = 25mm 122mm 35mm 22mm,clip,width=55mm]{EVFig.pdf}
\includepdf[pages=12,pagecommand={},offset=57mm 29mm,
trim = 25mm 122mm 35mm 22mm,clip,width=55mm]{EVFig.pdf}
\vskip77mm
\caption{\label{fig3} The curves $\Gamma_2(A)$ (solid) and $\Gamma_1(A)$ (dashed) 
for six real $5\times 5$ matrices. The vertical lines are $s=\delta_3$,
$s=\delta_2$ and $s=\delta_1$. }
\end{figure}

It is clear that in these figures $s=\delta_3$ is an asymptote to the curve $\Gamma_2(A)$,
and $s=\delta_2$ is an asymptote to $\Gamma_1(A)$, as stated in \cite{AT}. One can 
generalize this observation and prove it for arbitrary $k$.

\begin{theorem}\label{asy}
The curve $\Gamma_k(A)$ : $|\det W_k|^2 (s-\delta_{k+1}) =
 (\sigma_1(V_k))^2\lambda_{max}(H(\det W_k adj(W_k^*)))\ $
has $s=\delta_{k+1}$ as an asymptote, and there are no points on the curve
with $s<\delta_{k+1}$.
\end{theorem}
%\vskip.05cm

\begin{proof}
For $s<\delta_k$, $H(W_k)=diag(\delta_1-s, \dots , \delta_k-s)$ is positive definite which implies
that $H(W_k^{-1})$ is positive definite \cite{M}. Then $|\det W_k|>0$ and 
$\lambda_{max}(H(\det W_k adj(W_k^*)))>0$, so $s<\delta_{k+1}$ is not possible. Furthermore,
if $V_k\ne 0$ we see that $|\det W_k|\to\infty$ is needed as $s\to\delta_{k+1}^+$ and this implies
$|t|\to\infty$. If $V_k=0$ then $s=\delta_{k+1}$ is the curve $\Gamma_k(A)$ except for isolated points given by
$\det W_k=0$ (and $A$ is unitarily similar to a direct sum of $\Delta_k+Y_k$ and some 
$B\in {\cal  M}_{n-k}(\C)\,$).
\end{proof}

% For $k=1$ this result was proven in \cite{AT}.

%\smallskip

Let $k=2$ again. For $s<\delta_2$, $H(W_2^{-1})$ is positive definite which means that both
eigenvalues of $M_2=H(\det W_2 adj(W_2^*))$ are positive. By (\ref{eq:B2}) and (\ref{eq:B3}),
$2\lambda_{min}(M_2)=m_1+m_3-\sqrt{(m_1-m_3)^2+4|m_2|^2}=
Tr M_2 -\sqrt{(Tr M_2)^2-4\det M_2}$. 
Denote by $\gamma_2(A)$ the curve
\begin{equation}\label{eq:ga}
|\det W_2|^2 (s-\delta_{3}) = (\sigma_1(V_2))^2\lambda_{min}(M_2)\ .
\end{equation}
Then $\gamma_2(A)$ must be located to the left of $\Gamma_2(A)$ but still  to the right of 
$s=\delta_3$, and with $s=\delta_3$ as an asymptote.
With $K_2=(\sigma_1(V_2))^2$  and 
$2\lambda_{max}(M_2)=Tr M_2 +\sqrt{(Tr M_2)^2-4\det M_2}$, the curves $\Gamma_2(A)$ and 
$\gamma_2(A)$ are given by 
$\ 2|\det W_2|^2 (s-\delta_{3}) -K_2\, Tr M_2=\pm K_2\sqrt{(Tr M_2)^2-4\det M_2}\ $, with
the plus sign for $\Gamma_2(A)$ and the minus sign for $\gamma_2(A)$. Squaring implies
\begin{equation}\label{eq:pol}
4|\det W_2|^4 (s-\delta_{3})^2 -4K_2Tr M_2|\det W_2|^2 (s-\delta_{3})=-4K_2^2\det M_2\ ,
\end{equation}
which is a polynomial curve $\Gamma_2(A)\cup \gamma_2(A)$ for $s$ and $t$. 
The  components $\Gamma_2(A)$ and  $\gamma_2(A)$  of (\ref{eq:pol}) connect at points where 
$(Tr M_2)^2=4\det M_2$ and at infinity.

We illustrate the curve in (\ref{eq:pol}) in Figure \ref{fig4} for two random complex $5\times 5$ matrices, 
for which the two components $\Gamma_2(A)$ and $\gamma_2(A)$ are disjoint, and for the matrix
\begin{equation}\label{eq:Ahat}
\hat A=\begin{pmatrix} 2 &  0&0&-1.01  \\  0 &1& 0 &0 \\
0& 0 &0 & -1 \\ 1.01&0&1&0  \end{pmatrix}\ .
\end{equation}
We see that $s=\delta_3$ is an asymptote  for the two components. For $\hat A$
we see that the two components meet at some points and that they are non-smooth there. 
We shall comment more on this in the next section.

\begin{figure}[!ht]
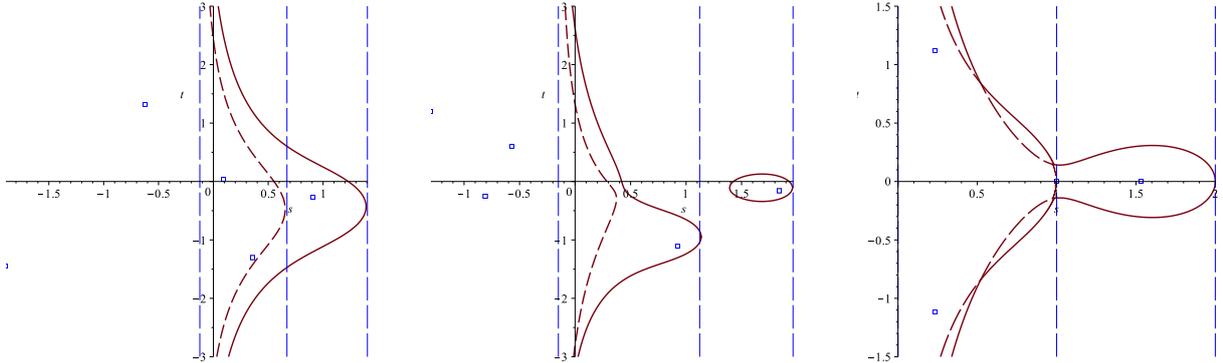

\includepdf[pages=13,pagecommand={},offset=-55mm -11mm,
trim = 25mm 122mm 35mm 22mm,clip,width=55mm]{EVFig.pdf}
\includepdf[pages=14,pagecommand={},offset=1mm -11mm,
trim = 25mm 122mm 35mm 22mm,clip,width=55mm]{EVFig.pdf}
\includepdf[pages=15,pagecommand={},offset=57mm -11mm,
trim = 25mm 122mm 35mm 22mm,clip,width=55mm]{EVFig.pdf}
\vskip36mm
\caption{\label{fig4} The two parts $\Gamma_2(A)$ (solid) and $\gamma_2(A)$ (dashed)
of the curve given by (\ref{eq:pol}), for two random complex $5\times 5$ matrices
(left and center), and for the matrix $\hat A$ in (\ref{eq:Ahat}) (right). 
The vertical lines are $s=\delta_3$, $s=\delta_2$ and $s=\delta_1$.}
\end{figure}

Although in Figures \ref{fig2} and \ref{fig3} the curves $\Gamma_2(A)$  give a strict 
improvement compared with the cubic curves $\Gamma_1(A)$, it is not uncommon 
that for some region in the band $\delta_2<s<\delta_1$, the cubic curve is better.

In Figure \ref{fig5} we show three cases, the first two subfigures refer to a complex and a real 
$5\times 5$ matrix, respectively, 
and the third to the real matrix $\tilde A$ in (\ref{eq:Atilde}),  where in 
some region the cubic curve $\Gamma_1(A)$ gives more restrictions on the spectrum than
$\Gamma_2(A)$. We see that we can even have a closed loop on $\Gamma_1(A)$ without
having a closed loop on $\Gamma_2(A)$.

A general analysis of this situation is complex but we can state some sufficient conditions for the
cubic curve $\Gamma_1(A)$ to be more restrictive. Assume for simplicity that  $A\in {\cal  M}_{n}(\R)$. 
Then the situation will occur
if for some value of $s$ the value of $|t|$ is smaller, or maybe non-existent, on $\Gamma_1(A)$ than 
on $\Gamma_2(A)$, since $s=\delta_2$ is an asymptote for $\Gamma_1(A)$. 

Assuming $A\in {\cal  M}_{n}(\R)$, we have $\alpha=0$ and $K_1=||S(A)\bold u_1||_2^2$, 
so by (\ref{eq:cub}) the cubic curve $\Gamma_1(A)$ becomes
\begin{equation*}
[(s-\delta_1)^2+t^2](s-\delta_2)= K_1(\delta_1-s)\ ,
\end{equation*}
which gives $t^2=K_1(\delta_1-s)/(s-\delta_2)- (s-\delta_1)^2$.
For $k=2$ the corresponding analysis is harder, we therefore choose to compare the values of $t^2$
on $\Gamma_1(A)$ and $\Gamma_2(A)$, denoted by $t_1^2$ and $t_2^2$, respectively,
for the value $\tilde s=(\delta_1+\delta_2)/2$ of $s$, and check when $t_1^2< t_2^2$. 
For $k=1$, on $\Gamma_1(A)$, we get $t_1^2=K_1-(\delta_1-\delta_2)^2/4$.
For $k=2$, on $\Gamma_2(A)$, with $\alpha=0$, $\beta=0$, $\gamma\in\R$ and $s=\tilde s$,
Proposition \ref{3H} gives 
\begin{equation*}
[(\frac{\delta_1-\delta_2}{2})^2+t_2^2-\gamma^2]^2(\frac{\delta_1+\delta_2}{2}-\delta_3)
= \frac{K_2}{2}|(\delta_1-\delta_2)[(\frac{\delta_1-\delta_2}{2})^2+t_2^2-\gamma^2]|\ ,
\end{equation*} 
where 
\begin{equation*}
2K_2=||S(A)\bold u_1||_2^2+||S(A)\bold u_2||_2^2-2\gamma^2+
\sqrt{(||S(A)\bold u_1||_2^2-||S(A)\bold u_2||_2^2)^2+4((S(A)\bold u_2)^*S(A)\bold u_1)^2}\ .
\end{equation*}

\begin{figure}
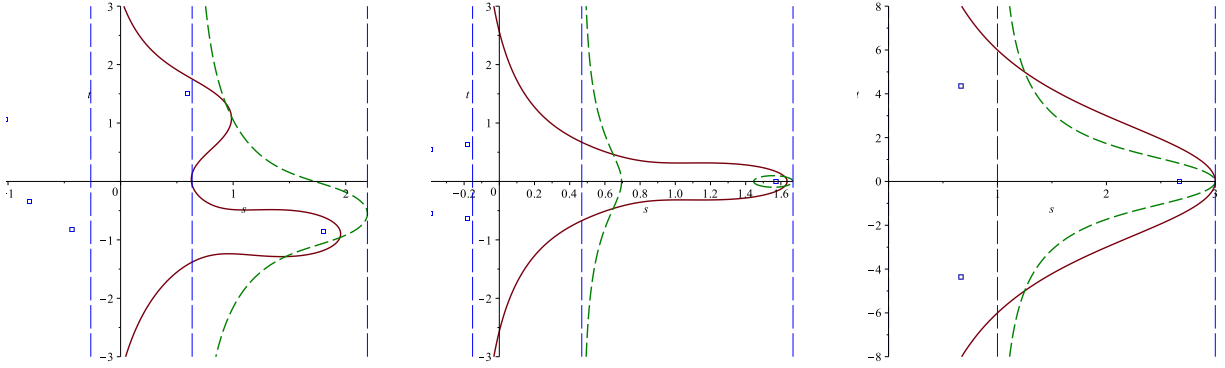

\includepdf[pages=16,pagecommand={},offset=-55mm 85mm,
trim = 25mm 122mm 35mm 22mm,clip,width=55mm]{EVFig.pdf}
\includepdf[pages=17,pagecommand={},offset=1mm 85mm,
trim = 25mm 122mm 35mm 22mm,clip,width=55mm]{EVFig.pdf}
\includepdf[pages=18,pagecommand={},offset=57mm 85mm,
trim = 25mm 122mm 35mm 22mm,clip,width=55mm]{EVFig.pdf}
\vskip35mm
\caption{\label{fig5} The curves $\Gamma_2(A)$ (solid) and $\Gamma_1(A)$ (dashed)
for a complex $5\times 5$ matrix (left), a real $5\times 5$ matrix (center), and the matrix 
$\tilde A$ in (\ref{eq:Atilde}) (right). The vertical lines are $s=\delta_3$,
$s=\delta_2$ and $s=\delta_1$.}
\end{figure}

This implies that $t_2^2=\gamma^2-(\frac{\delta_1-\delta_2}{2})^2$ or
$t_2^2=\gamma^2-(\frac{\delta_1-\delta_2}{2})^2\pm K_2(\delta_1-\delta_2)/(\delta_1+\delta_2-2\delta_3)$.
Since by (\ref{eq:V}), $K_1=\gamma^2+||\bold v_1||_2^2\ge\gamma^2$, we have 
$t_1^2\ge \gamma^2-(\frac{\delta_1-\delta_2}{2})^2\ge
 \gamma^2-(\frac{\delta_1-\delta_2}{2})^2-K_2(\delta_1-\delta_2)/(\delta_1+\delta_2-2\delta_3)$.
The question is therefore when $t_1^2< \gamma^2-(\frac{\delta_1-\delta_2}{2})^2+K_2(\delta_1-\delta_2)/(\delta_1+\delta_2-2\delta_3)$ is possible. Since $t_1^2=K_1-(\frac{\delta_1-\delta_2}{2})^2$, we need to have
$K_1< \gamma^2+K_2(\delta_1-\delta_2)/(\delta_1+\delta_2-2\delta_3)$.
Expressed in terms of the vectors $\bold v_1$ and $\bold v_2$ of (\ref{eq:V}), and using (\ref{eq:K2}), 
this inequality is
\begin{equation}\label{eq:cross}
||\bold v_1||_2^2<K_2\cdot \frac{\delta_1-\delta_2}{\delta_1+\delta_2-2\delta_3}\ ; \ 
2K_2=||\bold v_1||_2^2+||\bold v_2||_2^2+\sqrt{(||\bold v_1||_2^2-||\bold v_2||_2^2)^2+4(\bold v_2^*\bold v_1)^2}\ .
\end{equation}
If the vectors $\bold v_1$ and $\bold v_2$ are orthogonal, $K_2=max\{||\bold v_1||_2^2,||\bold v_2||_2^2\}$, if they are
parallel, $K_2=||\bold v_1||_2^2+||\bold v_2||_2^2$. Also, $0<(\delta_1-\delta_2)/(\delta_1+\delta_2-2\delta_3)=
(\delta_1-\delta_2)/(\delta_1-\delta_2+2(\delta_2-\delta_3))<1$ and closer to 1 if $\delta_2-\delta_3$
is small compared to $\delta_1-\delta_2$. Thus matrices with $||\bold v_1||_2$ small compared to 
$||\bold v_2||_2$ and $\delta_2-\delta_3$ small compared to $\delta_1-\delta_2$ will have the desired property.

The $3\times 3$ matrix
\begin{equation}\label{eq:Atilde}
\tilde A=\begin{pmatrix} 3 &  0& -2  \\  
0& 1 & -4 \\ 2&4&0  \end{pmatrix}
\end{equation}
has $\delta_1=3$, $\delta_2=1$, $\delta_3=0$, $\bold v_1,\bold v_2\in\R$, $||\bold v_1||_2=2$, 
$||\bold v_2||_2=4$, $\bold v_2^*\bold v_1= 8$, $K_1=4$, and $K_2=20$. We get $\tilde s=2$
and $t_1^2=3<9=t_2^2$, 
and we see the curves in the last subfigure of Figure \ref{fig5}.

In these cases, one may of course combine $\Gamma_1(A)$ and $\Gamma_2(A)$
and use the intersection of the two regions to minimize the region for the spectrum. 
We also emphasize that the condition (\ref{eq:cross}) is sufficient but not necessary for
the two curves to cross since we assumed $A$ real and $s=\tilde s=(\delta_1+\delta_2)/2$.

%----------------------------------------------------------------------------------------------------------------------

\section{Topologically interesting examples}

To analyze the curves $\Gamma_k(A)$ for larger values of $k$ is in general hard, 
the complicated dependence of $\lambda_{max}(H(\det W_k adj(W_k^*)))\ $ 
on $s$ and $t$ being a main cause. Each element of $adj(W_k^*)$ is a $(k-1)\times (k-1)$ minor of $W_k^*$.
If we assume that we have a matrix $A$ such that $W_k$ is diagonal, then the curves
are easier to analyze and we can already in this case see interesting types of behavior
of $\Gamma_k(A)$  that can occur. Using the notation of (\ref{eq:21}) and (\ref{eq:22}),
we therefore assume that $A$ has the form

\begin{equation}\label{eq:AWdiag}
A=\begin{pmatrix} \Delta_k &  -V_k^*  \\  V_k & \tilde A_k  \end{pmatrix}\ ,
\end{equation}
where $\Delta_k=diag(\delta_{1}, \dots ,\delta_{k})\in {\cal  M}_{k}(\R)$ and 
$\tilde A_k=\tilde \Delta_k+\tilde Y_k$. This means that in (\ref{eq:21}) and (\ref{eq:22})
we have assumed $Y_k=0$ and $U=I_n$. Then, by (\ref{eq:23}),
$W_k=\Delta_k-\lambda I_k$ is diagonal with 
\begin{equation*}
\det W_k=\prod_{r=1}^{k}(\delta_r-\lambda)\ .
\end{equation*}
We get
\begin{eqnarray*}
\det W_k adj(W_k^*)=\det W_k adj[diag(\delta_{1}-\bar\lambda, \dots ,\delta_{k}-\bar\lambda)] 
\hskip28mm\\
=\det W_k diag(\prod_{r=2}^{k}(\delta_{r}-\bar\lambda), \dots , \prod_{r=1}^{k-1}(\delta_{r}-\bar\lambda)) 
\hskip20mm\\
=diag((\delta_{1}-\lambda)\prod_{r=2}^{k}|\delta_{r}-\lambda|^2, \dots , 
(\delta_{k}-\lambda)\prod_{r=1}^{k-1}|\delta_{r}-\lambda|^2) \ .
\end{eqnarray*}
With $\lambda=s+it$, the Hermitian part is
\begin{equation*}
M_k=H(\det W_k adj(W_k^*))=diag((\delta_{1}-s)\prod_{r=2}^{k}|\delta_{r}-\lambda|^2, \dots , 
(\delta_{k}-s)\prod_{r=1}^{k-1}|\delta_{r}-\lambda|^2)\ .
\end{equation*}
The eigenvalues of $M_k$ are its diagonal elements and we need to determine the largest
eigenvalue. Suppose that $j<i$, so $\delta_j\ge\delta_i$, and compare the corresponding
diagonal elements $(M_k)_{jj}$ and $(M_k)_{ii}$ of $M_k$. We have
$(\delta_{j}-s)\prod_{r\ne j}|\delta_{r}-\lambda|^2\ge
 (\delta_{i}-s)\prod_{r\ne i}|\delta_{r}-\lambda|^2\ $ if 
 $(\delta_{j}-s)|\delta_{i}-\lambda|^2\ge(\delta_{i}-s)|\delta_{j}-\lambda|^2\ $. This means
 $ (\delta_{j}-s)((\delta_{i}-s)^2+t^2)\ge(\delta_{i}-s)((\delta_{j}-s)^2+t^2)$, or,
 equivalently $(\delta_{j}-s)(\delta_{i}-s)\le t^2$.
 Writing it as
 \begin{equation*}
(s-\frac{\delta_j+\delta_i}{2})^2-t^2\le(\frac{\delta_j-\delta_i}{2})^2\ ,
\end{equation*}
we see that equality is attained at the hyperbola with centerpoint at
$(\frac{\delta_j+\delta_i}{2})^2$ on the real axis, passing through the real axis at
$\delta_j$ and $\delta_i$, and with asymptotes of slope $\pm 1$.

Suppose we are in the region $\delta_{k+1}\le s\le \delta_1$. No components bending in
the same direction (left or right) of all the hyperbolas for all pairs $j,i$ will cross each other since
they have the same asymptotic slope. Of all hyperbolas formed from $\delta_1$ and
$\delta_j$, $j=2, \dots , k$, it is the left component of the one formed from $\delta_1$ and
$\delta_2$ which gives the left boundary of the region where the first diagonal element $(M_k)_{11}$ of $M_k$
equals $\lambda_{max}(M_k)$. Next, to the left of this curve, the second diagonal element $(M_k)_{22}$
will be $\lambda_{max}(M_k)$ until we reach the left component  of
 the hyperbola formed from $\delta_2$ and $\delta_3$ and so on. We illustrate the
 situation in Figure \ref{fig6}. In the first subfigure we show all six hyperbolas in a case with
 $k=3$, $\delta_1=5$, $\delta_2=3.5$, $\delta_3=1$ and $\delta_4=0$. The hyperbolas
 related to $(\delta_1,\delta_2)$, $(\delta_1,\delta_3)$ and $(\delta_1,\delta_4)$ are solid, 
 the ones of $(\delta_2,\delta_3)$ and $(\delta_2,\delta_4)$ dashed, and the one of
 $(\delta_3,\delta_4)$ dotted.
 In the second subfigure, we have $k=4$, $\delta_1=5$, $\delta_2=3.5$, $\delta_3=3$,
 $\delta_4=1$  and $\delta_5=0$, and we show those components of all 10 hyperbolas
 that separate four regions of different diagonal elements of $M_4$ being 
 $\lambda_{max}(M_4)$ for $\delta_5\le s \le\delta_1$.

 \begin{figure}[!ht]
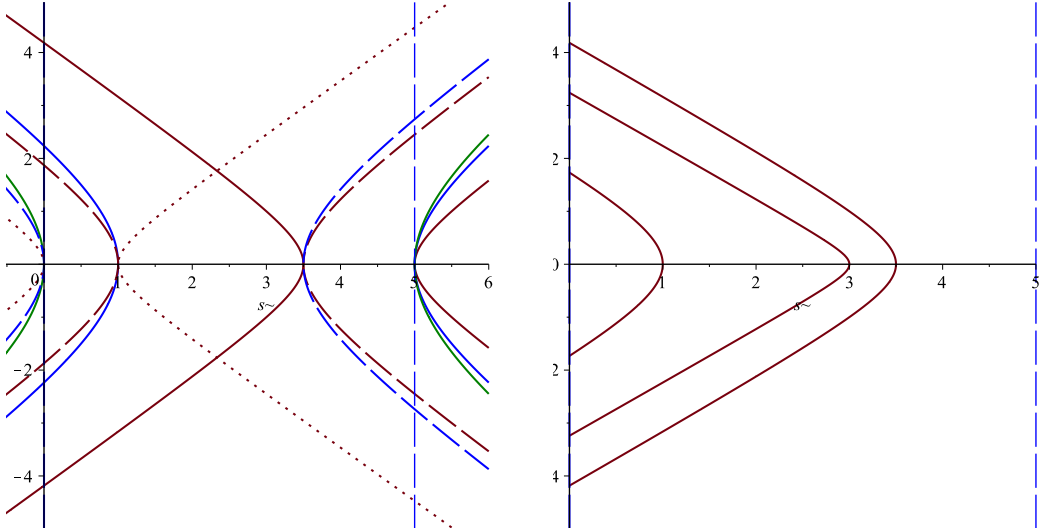

 \includepdf[pages=19,pagecommand={},offset=-30mm 5mm,
trim = 35mm 122mm 35mm 22mm,clip,width=75mm]{EVFig.pdf}
\includepdf[pages=20,pagecommand={},offset=42mm 5mm,
trim = 35mm 122mm 35mm 22mm,clip,width=75mm]{EVFig.pdf}
\vskip62mm
\caption{\label{fig6} All six hyperbolas for a $k=3$ case (left) and the three region separating
hyperbola components for a $k=4$ case (right).}
\end{figure}
 
 With $\sigma_1(V_k)=\varepsilon$, the curve $\Gamma_k(A)$ of Theorem \ref{NH} in 
 each region becomes
\begin{equation*}
\prod_{r=1}^{k}|\delta_r-\lambda|^2(s-\delta_{k+1})=
\varepsilon^2(\delta_j-s)\prod_{r=1, r\ne j}^{k}|\delta_{r}-\lambda|^2\ ,
\end{equation*}
which is $|\delta_j-\lambda|^2(s-\delta_{k+1})=\varepsilon^2(\delta_j-s)$ or
\begin{equation}\label{eq:eqdiag}
[(\delta_j-s)^2+t^2](s-\delta_{k+1})=\varepsilon^2(\delta_j-s) \ ,
\end{equation}
which is a curve of degree 3.
Thus, our entire curve $\Gamma_k(A)$ consists of different cubic curves which are connected in a 
continuous but maybe non-smooth way at the hyperbolas that separates the regions.
Note that no details of $V_k$ or $\tilde A_k$ in (\ref{eq:AWdiag}) are needed to construct the curve.

Putting $t=0$ in (\ref{eq:eqdiag}), we observe that $s=\delta_j$ or
\begin{equation*}
s=s_{\pm}=\frac{\delta_j+\delta_{k+1}}{2}\pm\sqrt{(\frac{\delta_j-\delta_{k+1}}{2})^2-
\varepsilon^2}\ .
\end{equation*}
For $1\le j\le k-1$, the value $s_{+}$ is between $\delta_{j+1}$ and $\delta_j$ if 
$s_{+}\ge\delta_{j+1}$, which happens if $\varepsilon\le\varepsilon_j$, where
\begin{equation}\label{eq:eps}
\varepsilon_j=\sqrt{(\delta_{j+1}-\delta_{k+1})(\delta_{j}-\delta_{j+1})}\ .
\end{equation}
Note that $\varepsilon_j \le ((\delta_{j+1}-\delta_{k+1})+(\delta_{j}-\delta_{j+1}))/2
=(\delta_j-\delta_{k+1})/2$ so $s_{\pm}$ exist (real) for $\varepsilon\le\varepsilon_j$.
Then there will be a closed loop passing through $s=s_{+}$ and $s=\delta_j$.
If also $s_{-}\ge\delta_{j+1}$ (which could only happen if $\delta_{j+1}\le
(\delta_{j}+\delta_{k+1})/2$), then $s_{-}$ will be on the unbounded
component of  $\Gamma_k(A)$ (and any loops to the left have already merged with 
the unbounded component).

In the region most to the left, where $s\le\delta_k$, for 
$\varepsilon<\varepsilon_k=(\delta_k-\delta_{k+1})/2$ we have 
$s_{\pm}>\delta_{k+1}$ and there is a loop through $\delta_k$ and $s_{+}$ while
the unbounded component  of $\Gamma_k(A)$ passes through $s=s_{-}$. 
For $\varepsilon=\varepsilon_k$ the loop
connects to the unbounded component. For $\varepsilon>\varepsilon_k$ there is
only one component,  it is unbounded and passes through $s=\delta_k$.

In general, because of the different hyperbolas, the unbounded component of $\Gamma_k(A)$ 
passes through several regions as $|t|$ grows and $s\to\delta_{k+1}^{+}$, so it may 
have points where it is not smooth. 

An example of this is provided by the matrix $\hat A$ in (\ref{eq:Ahat}). 
In the last subfigure of Figure \ref{fig4} we see the non-smoothness of $\Gamma_2(\hat A)$ 
which occurs at 
points of the hyperbola $(s-3/2)^2+t^2=1/4$. We have $\varepsilon=1.01$ and the four crossings 
where $\lambda_{min}(M_2)=\lambda_{max}(M_2)$, so $\Gamma_2(\hat A)$ and $\gamma_2(\hat A)$
of (\ref{eq:ga}) meet, are readily calculated to have coordinates
$s=s_{1,2}=(3\pm\sqrt{9-8\varepsilon^2})/4\approx 0.75\pm 0.229$, each with two
corresponding $t$-values $t=\pm\sqrt{s_{1,2}^2-3s_{1,2}+2}$. 
For  $s\in\  [0,0.521]\  \cup\  [0.979,2] $ the diagonal element
$(M_2)_{11}=(2-s)[(1-s)^2+t^2]$ of $M_2$ is its largest eigenvalue, and for  
$s\in\  [0.521,0.979] $,  $(M_2)_{22}=(1-s)[(2-s)^2+t^2]$ is the largest,
and these non-smooth changes take place at the crossing with the hyperbola.

\medskip

With the above relations between $\delta_1,\dots, \delta_{k+1}$ and 
$\varepsilon_1,\dots, \varepsilon_k$, one can construct matrices that give curves
$\Gamma_k(A)$ with desired topologies.

As a first example, let $k=2$ and define two matrices
\begin{equation}\label{eq:AoB}
A=\begin{pmatrix} 1.36 &  0&0&-\varepsilon/2  \\  0 &1& -\varepsilon &0 \\
0& \varepsilon &0 & -0.25 \\ \varepsilon/2&0&0.25&0  \end{pmatrix},\ 
B=\begin{pmatrix} 1.16 &  0&0&-\varepsilon/2  \\  0 &1& -\varepsilon &0 \\
0& \varepsilon &0 & -0.25 \\ \varepsilon/2&0&0.25&0  \end{pmatrix}
\end{equation}
Both have $\delta_3=0$, $\delta_2=1$, and $\sigma_1(V_2)=\varepsilon$. Then 
$\varepsilon_2=(\delta_2-\delta_{3})/2=0.5$, and, by (\ref{eq:eps}), 
$\varepsilon_1=\sqrt{\delta_1-1}$.
For $A$, $\varepsilon_1=0.6$, and for $B$, $\varepsilon_1=0.4$. For small
$\varepsilon$, both matrices will have curves $\Gamma_2(A)$, $\Gamma_2(B)$ with 
two closed loops enclosing one eigenvalue each.
For $A$, when $\varepsilon$ increases, first at $\varepsilon=0.5$ the left loop
merges with the unbounded component, and then at $\varepsilon=0.6$ the remaining
 loop merges with the unbounded component.
For $B$, when $\varepsilon$ increases, first at $\varepsilon=0.4$ the two loops 
merge into one that encloses two eigenvalues, and at $\varepsilon=0.5$ this loop
merges with the unbounded component.
 In Figure \ref{fig7} we show first $\Gamma_2(A)$ for A for $\varepsilon=0.45, 0.55, 0.65$ and
 below $\Gamma_2(B)$ for $B$ for $\varepsilon=0.35, 0.45, 0.55$.
 The eigenvalues are marked by small boxes.
 In the last plot of $A$ and in the last two plots of $B$, one sees clearly the non-smooth
 points on the curves which are located at the crossings with the left-bending component of the 
 hyperbola that passes through $s=1, \, t=0$.

For larger values of $k$, one can also choose values of $\delta_1,\dots ,\delta_{k+1}$
such that the mergers of neighboring loops of $\Gamma_k(A)$ come in any desired order.
If we want all mergers to occur for the same $\varepsilon$, i.e., if
$\varepsilon_1=\varepsilon_2=\dots =\varepsilon_k=(\delta_k-\delta_{k+1})/2$, we can also 
obtain this. By (\ref{eq:eps}) this will happen if, given $\delta_{k+1}$ and $\delta_k$,
we choose iteratively for decreasing $j=k,\dots, 2$, $\delta_{j-1}=\delta_j+
(\delta_k-\delta_{k+1})^2/(4(\delta_j-\delta_{k+1}))$. 
As an example, we take $k=4$, $\delta_5=0$ and $\delta_4=1$. Then
we get $\delta_3=5/4$, $\delta_2=29/20$ and $\delta_1=941/580$,
and the simultaneous merger is at $\varepsilon=1/2$. Define a matrix
\begin{equation}\label{eq:MC}
C=\begin{pmatrix} 941/580 &  0&0&0 &0&-\varepsilon/\sqrt{2} \\  0 &29/20& 0&0&-\varepsilon/\sqrt{2} &0 \\
0&0&5/4&0&0&-\varepsilon/\sqrt{2} \\ 0&0&0&1&-\varepsilon/\sqrt{2}&0 \\
0&\varepsilon/\sqrt{2}&0& \varepsilon/\sqrt{2} &0 & -0.25 \\ 
\varepsilon/\sqrt{2}&0&\varepsilon/\sqrt{2}&0&0.25&0  \end{pmatrix},
\end{equation}
which, since $V_4\in {\cal  M}_{2,4}(\R)$ has $\sigma_1(V_4)=\varepsilon$, has the desired properties. 
In Figure \ref{fig8} we depict the curve $\Gamma_4(C)$ of $C$ for $\varepsilon=0.45, 0.5, 0.55$.
In the last subfigure we see the non-smooth points on the curve located at the crossings with the hyperbolas.

\begin{figure}
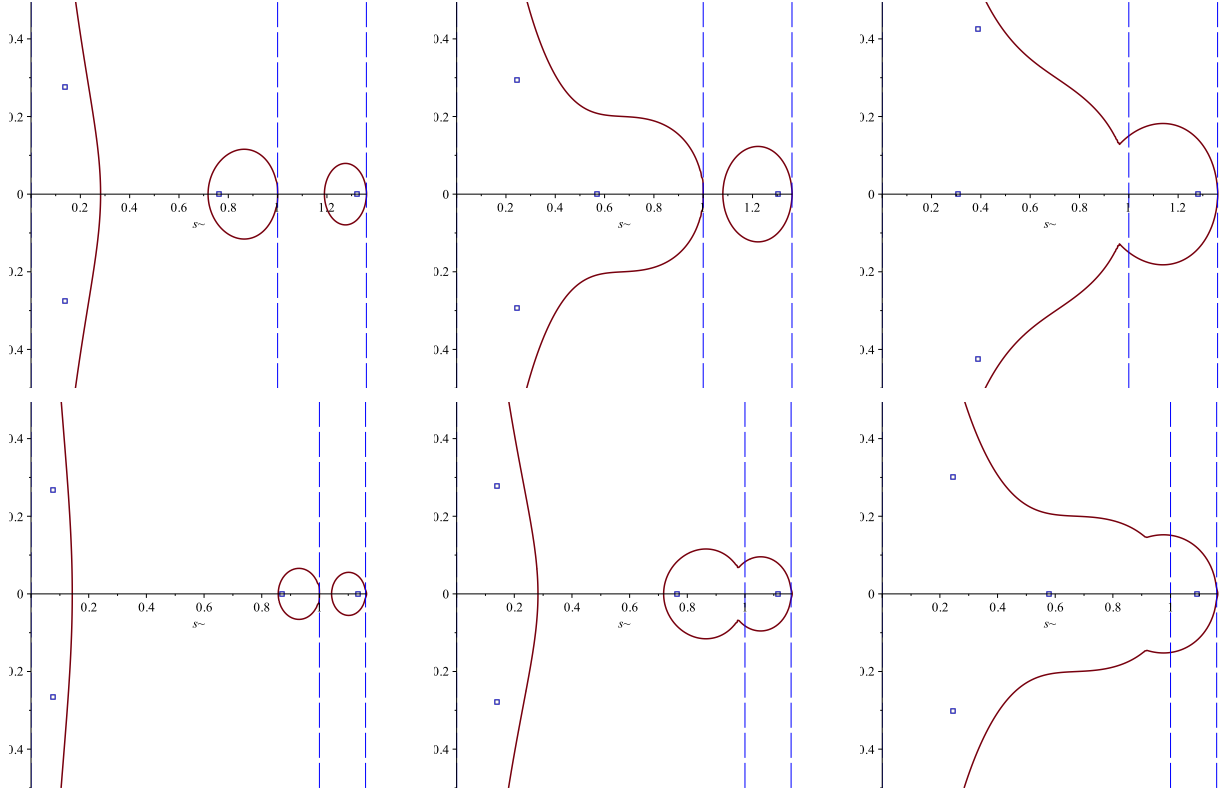

\includepdf[pages=21,pagecommand={},offset=-55mm 85mm,
trim = 35mm 122mm 35mm 22mm,clip,width=55mm]{EVFig.pdf}
\includepdf[pages=22,pagecommand={},offset=1mm 85mm,
trim = 35mm 122mm 35mm 22mm,clip,width=55mm]{EVFig.pdf}
\includepdf[pages=23,pagecommand={},offset=57mm 85mm,
trim = 35mm 122mm 35mm 22mm,clip,width=55mm]{EVFig.pdf}
\includepdf[pages=24,pagecommand={},offset=-55mm 32mm,
trim = 35mm 122mm 35mm 22mm,clip,width=55mm]{EVFig.pdf}
\includepdf[pages=25,pagecommand={},offset=1mm 32mm,
trim = 35mm 122mm 35mm 22mm,clip,width=55mm]{EVFig.pdf}
\includepdf[pages=26,pagecommand={},offset=57mm 32mm,
trim = 35mm 122mm 35mm 22mm,clip,width=55mm]{EVFig.pdf}
\vskip77mm
\caption{\label{fig7} The curves $\Gamma_2(A)$ (top) and $\Gamma_2(B)$ (bottom) for the 
matrices $A$ and $B$ in  (\ref{eq:AoB}) for increasing values of $\varepsilon$. }
\end{figure}

\medskip

Finally we demonstrate an interesting possibility. Let $k=2$ and
\begin{equation}\label{eq:MF}
F(\varepsilon_1 , \varepsilon_2)=
\begin{pmatrix} 5 &  -\varepsilon_2&0&0  \\  \varepsilon_2 &5& -\varepsilon_1 &0 \\
0& \varepsilon_1 &0 & -1 \\ 0&0&1&0  \end{pmatrix}\ .
\end{equation}

\newpage

\begin{figure}[h]
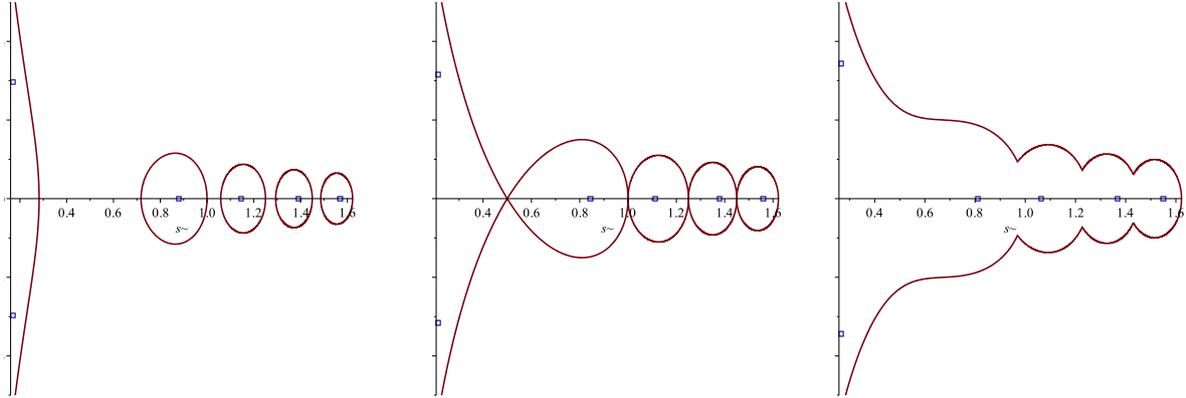

\includepdf[pages=27,pagecommand={},offset=-52mm 85mm,
trim = 40mm 122mm 33mm 22mm,clip,width=55mm]{EVFig.pdf}
\includepdf[pages=28,pagecommand={},offset=4mm 85mm,
trim = 40mm 122mm 33mm 22mm,clip,width=55mm]{EVFig.pdf}
\includepdf[pages=29,pagecommand={},offset=57mm 85mm,
trim = 40mm 122mm 33mm 22mm,clip,width=55mm]{EVFig.pdf}
\vskip40mm
\caption{\label{fig8} The curves $\Gamma_4(C)$ for the matrix $C$ in (\ref{eq:MC}) for increasing values 
of $\varepsilon$. }
\end{figure}

We start with $F(2.48,1.0)$  which has a pair of complex conjugated eigenvalues, 
each inside a loop of $\Gamma_2(F(2.48,1.0))$. If we decrease $\varepsilon_2$, the two loops will merge and form a new loop enclosing two eigenvalues. 
If instead we increase $\varepsilon_1$, each loop
will connect to the unbounded component first.
One can balance the parameters and obtain a situation where the loops approach each other and 
the unbounded component in such a way that an inner loop is formed that encloses a domain where 
no eigenvalue can be located. This is illustrated in Figure \ref{fig9} where the matrices used are
$F(2.48,1.0)$, $F(2.48,0.66)$, $F(2.52,1.0)$ and $F(2.52,0.66)$. On the top row the curves
$\Gamma_2(F(\varepsilon_1 , \varepsilon_2))$ are illustrated, and on the bottom row
the curves  $\gamma_2(F(\varepsilon_1 , \varepsilon_2))$ of (\ref{eq:ga}), obtained by 
using $\lambda_{min}(M_2)$, are added (dashed). The two curves together form a smooth 
algebraic curve given by (\ref{eq:pol}), but break up at non-smooth points for each component. 
We observe that the curve $\gamma_2(A)$ in (\ref{eq:ga}) may be non-connected and 
have a closed loop, compare with Figure \ref{fig4}.

\begin{figure}[b]
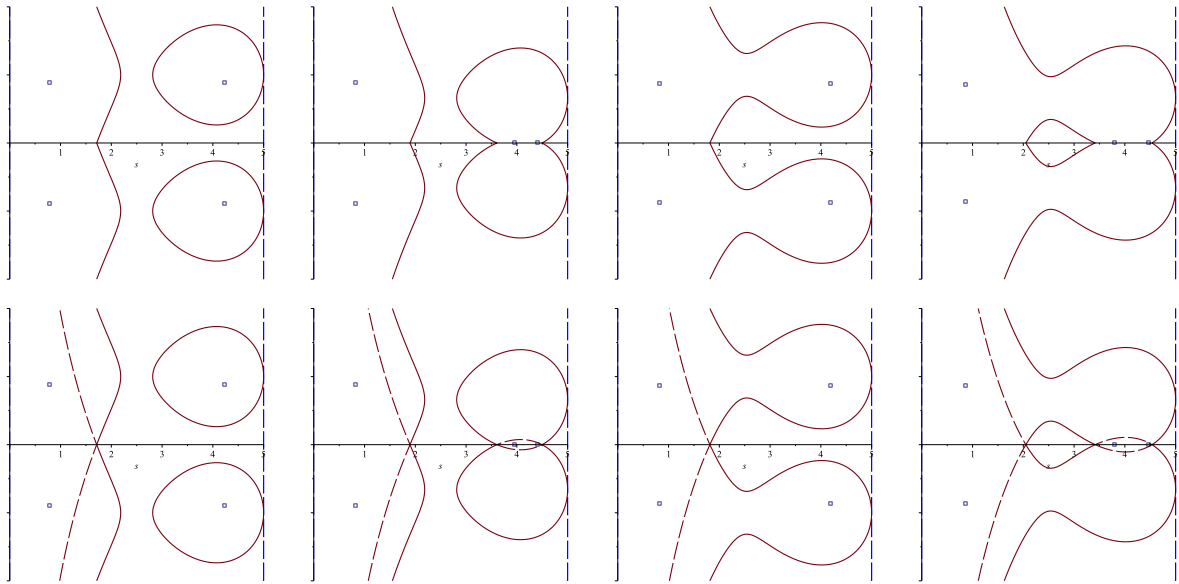

\includepdf[pages=30,pagecommand={},offset=-58mm -32mm,
trim = 35mm 122mm 35mm 22mm,clip,width=40mm]{EVFig.pdf}
\includepdf[pages=31,pagecommand={},offset=-18mm -32mm,
trim = 35mm 122mm 35mm 22mm,clip,width=40mm]{EVFig.pdf}
\includepdf[pages=32,pagecommand={},offset=22mm -32mm,
trim = 35mm 122mm 35mm 22mm,clip,width=40mm]{EVFig.pdf}
\includepdf[pages=33,pagecommand={},offset=62mm -32mm,
trim = 35mm 122mm 35mm 22mm,clip,width=40mm]{EVFig.pdf}
\includepdf[pages=34,pagecommand={},offset=-58mm -72mm,
trim = 35mm 122mm 35mm 22mm,clip,width=40mm]{EVFig.pdf}
\includepdf[pages=35,pagecommand={},offset=-18mm -72mm,
trim = 35mm 122mm 35mm 22mm,clip,width=40mm]{EVFig.pdf}
\includepdf[pages=36,pagecommand={},offset=22mm -72mm,
trim = 35mm 122mm 35mm 22mm,clip,width=40mm]{EVFig.pdf}
\includepdf[pages=37,pagecommand={},offset=62mm -72mm,
trim = 35mm 122mm 35mm 22mm,clip,width=40mm]{EVFig.pdf}
\vskip0mm
\caption{\label{fig9} The curves $\Gamma_2(F(\varepsilon_1 , \varepsilon_2))$ 
for the matrices $F(2.48,1.0)$, $F(2.48,0.66)$, 
$F(2.52,1.0)$ and $F(2.52,0.66)$ in (\ref{eq:MF}) (first row). 
For $F(2.52,0.66)$, no eigenvalue can be inside the closed loop. On the second row of plots,
the curves $\gamma_2(F(\varepsilon_1 , \varepsilon_2))$ are added (dashed).}
\end{figure}

%----------------------------------------------------------------------------------------------------------------------
\vskip3cm
\newpage

\section{Envelopes for the spectrum}

In the same way as for the lines used to bound the numerical range $F(A)$  \cite{HJ2,J} 
or the cubic curves $\Gamma_1(A)$ by Adam, Psarrakos and Tsatsomeros \cite{AT,PT1,PT2}, 
we can apply our theorem to
$e^{i\theta}A$ and obtain a curve that bounds $\sigma(e^{i\theta}A)$. Rotating that curve
by $e^{-i\theta}$, we get a curve that bounds $\sigma(A)$. Doing this for all
$\theta\in [0,2\pi[$ we get an infinite intersection of regions that contains $\sigma(A)$, and
whose boundary is the envelope of such curves.
\medskip

For a given matrix $A$, let $E_k(e^{i\theta}A)$ be all points $\lambda\in\C$ that satisfy
\begin{equation*}
|\det W_k|^2 (Re(\lambda)-\delta_{k+1}) \le (\sigma_1(V_k))^2\lambda_{max}(H(\det W_k adj(W_k^*)))\ ,
\end{equation*}
where $W_k=W_k(\theta)$, $V_k=V_k(\theta)$ and $\delta_{k+1}=\delta_{k+1}(\theta)$ are 
constructed from $e^{i\theta}A$ as in Section 2.

\medskip

We denote the region that is the intersection over all $\theta\in [0,2\pi[$ by $\mathcal{E}_k(A)$:
\begin{definition}
${\displaystyle{\mathcal{E}_k(A)=\bigcap_{\theta\in [0,2\pi[}e^{-i\theta}E_k(e^{i\theta}A)}}$
\end{definition}

The region $\mathcal{E}_1(A)$ was studied in detail in \cite{PT1,PT2} and we generalize
some of their results.
Recall that the $\ell $-rank numerical range of $A$ is 
$\Lambda_\ell(A) = \{\mu\in\C; X^*AX=\mu I_\ell, X \in {\cal  M}_{n,\ell}(\C), X^*X=I_\ell \} =
\bigcap_{\theta\in [0,2\pi[} e^{-i\theta}\{s+it; s\le\delta_\ell(e^{i\theta}A)\}$,
and that in general  $\sigma(A)\not\subset \Lambda_\ell(A)$ if $\ell \ge 2$  \cite{PT2}.

\begin{proposition}\label{Ek}
Let $A\in {\cal  M}_{n}(\C)$. For the regions $\mathcal{E}_k(A)$ of the complex plane the following hold: 

\smallskip
(i) $\ \ \ \sigma(A)\subset  \mathcal{E}_k(A)$

\smallskip
(ii) $\ \ \mathcal{E}_k(\tilde U^* A\tilde U)=\mathcal{E}_k(A)$ if $\tilde U$ is unitary

\smallskip
(iii) $\ \, \mathcal{E}_k(A^T)=\mathcal{E}_k(A)$

\smallskip
(iv) $\ \ \mathcal{E}_k(A^*)=\overline{\mathcal{E}_k(A)}$

\smallskip
(v) $\ \ \ \mathcal{E}_k(aA+bI_n)=a\mathcal{E}_k(A)+b\ $ if  $a,b\in\C$

\smallskip
(vi) $\ \ \Lambda_{k+1}(A)\subset  \mathcal{E}_k(A)$, where $\Lambda_{k+1}(A)$ is the
$(k+1)$-rank numerical range of $A$

\end{proposition}

\begin{proof}
Statement $(i)$ is clear from above, and $(ii)$ - $(v)$ follow directly from 
Proposition \ref{Gk}. $a\in\C$ is allowed since we use all rotations $e^{i\theta}A$ to 
define $\mathcal{E}_k(A)$.
By Theorem \ref{asy}, $\Gamma_k(e^{i\theta}A)\subset 
\{s+it; s\ge\delta_{k+1}(e^{i\theta}A)\}$ which implies 
$\{s+it; s\le\delta_{k+1}(e^{i\theta}A)\}\subset E_k(e^{i\theta}A)$.
Thus $\Lambda_{k+1}(A)=\bigcap_{\theta\in [0,2\pi[} e^{-i\theta}\{s+it; s\le\delta_{k+1}(e^{i\theta}A)\}
\subset \bigcap_{\theta\in [0,2\pi[}e^{-i\theta}E_k(e^{i\theta}A)= \mathcal{E}_k(A)$
which proves $(vi)$.
\end{proof}

We now give several illustrations of $\mathcal{E}_2(A)$. All plots are made using 120 curves, 
separated by 3 degrees,
that is, $\theta= 2\pi m/120, m=0, \dots ,119$.
In Figure  \ref{fig1}  the regions $F(A)$, $\mathcal{E}_1(A)$ and
$\mathcal{E}_2(A)$  for the Toeplitz matrix in (\ref{eq:Toep}) are illustrated. This matrix was used as an example
in \cite{PT1} to illustrate $\mathcal{E}_1(A)$. 

\medskip

On the top row of  Figure \ref{fig17} we show $F(A_1)$, $\mathcal{E}_1(A_1)$ and 
$\mathcal{E}_2(A_1)$ for the matrix
\begin{equation}\label{eq:A1}
A_1=\begin{pmatrix} 14+19i &  -4-i&-55-13i&-32+13i  \\  27+2i &14-25i& 64 &72 \\
54+i& 47-3i &14+44i & -32-42i \\ 76&73&4-2i&-11+24i  \end{pmatrix}
\end{equation}
that was used as an example in \cite{PT2}. Here we observe that $\mathcal{E}_1(A_1)$  isolates one 
and $\mathcal{E}_2(A_1)$ two of the eigenvalues, which are marked by small boxes.

On the middle row of Figure \ref{fig17} are the corresponding regions for the $11\times 11$ Frank 
matrix $A_2$ used in \cite{PT1} 
to illustrate $\mathcal{E}_1(A_2)$. Frank matrices have elements 
$A_{ij}=0$ if $j\le i-2$, $A_{ij}=n+1-i$ if $j=i-1$,
and $A_{ij}=n+1-j$ if $j\ge i$. They have determinant 1 but are ill-conditioned. 

On the bottom row of Figure \ref{fig17} are the corresponding regions for a real random 
$5\times 5$ matrix  $A_3$ with elements
between -1 and 1, and 5 real eigenvalues, $\mathcal{E}_1(A_3)$  isolates two of them 
and $\mathcal{E}_2(A_3)$ three.

\vskip2mm

\begin{figure}[!ht]
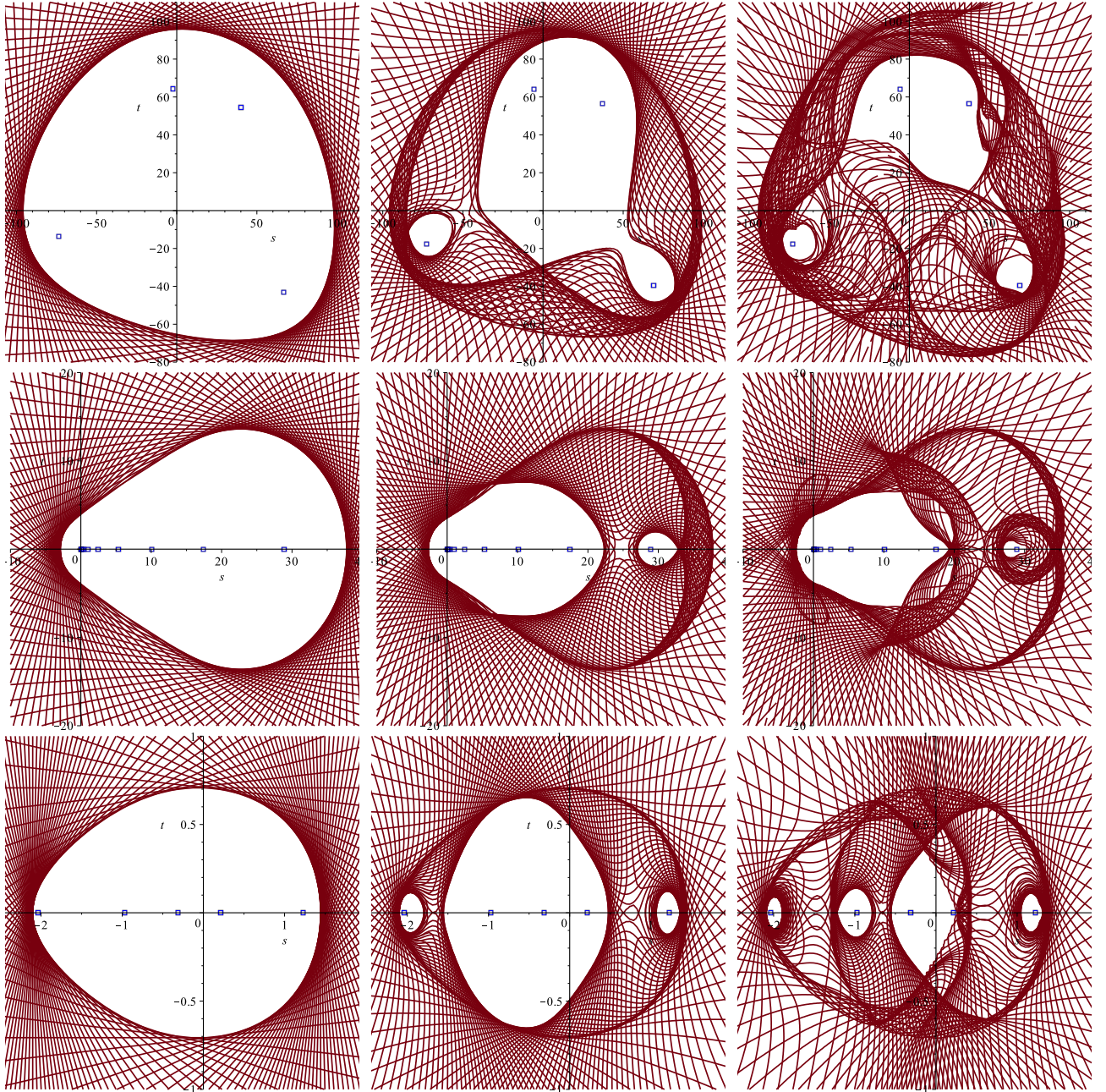

\vskip0cm
\includepdf[pages=38,pagecommand={},offset=-55mm 50mm,
trim = 25mm 120mm 60mm 20mm,clip,width=55mm]{EVFig.pdf}
\includepdf[pages=39,pagecommand={},offset=2mm 50mm,
trim = 25mm 120mm 60mm 20mm,clip,width=55mm]{EVFig.pdf}
\includepdf[pages=40,pagecommand={},offset=59mm 50mm,
trim = 25mm 120mm 60mm 20mm,clip,width=55mm]{EVFig.pdf}
\includepdf[pages=41,pagecommand={},offset=-55mm  -7mm,
trim = 25mm 120mm 60mm 20mm,clip,width=55mm]{EVFig.pdf}
\includepdf[pages=42,pagecommand={},offset=2mm  -7mm,
trim = 25mm 120mm 60mm 20mm,clip,width=55mm]{EVFig.pdf}
\includepdf[pages=43,pagecommand={},offset=59mm  -7mm,
trim = 25mm 120mm 60mm 20mm,clip,width=55mm]{EVFig.pdf}
\includepdf[pages=44,pagecommand={},offset=-55mm -64mm,
trim = 25mm 120mm 60mm 20mm,clip,width=55mm]{EVFig.pdf}
\includepdf[pages=45,pagecommand={},offset=2mm -64mm,
trim = 25mm 120mm 60mm 20mm,clip,width=55mm]{EVFig.pdf}
\includepdf[pages=46,pagecommand={},offset=59mm -64mm,
trim = 25mm 120mm 60mm 20mm,clip,width=55mm]{EVFig.pdf}
\vskip125mm
\caption{\label{fig17}The numerical range $F(A)$ (left), and the regions $\mathcal{E}_1(A)$ 
(center) and $\mathcal{E}_2(A)$ (right) of: 
 the  matrix $A_1$ in (\ref{eq:A1}) (top), 
 the $11\times 11$ Frank matrix $A_2$ in \cite{PT1} (middle), 
and  a real random  $5\times 5$  matrix $A_3$ (bottom).}
\vskip0.0cm
\end{figure}

\newpage

Finally, we can observe from simulations that $\mathcal{E}_1(A)$ can exclude regions of
the complex plane not excluded by  $\mathcal{E}_2(A)$, 
$\mathcal{E}_2(A)\not\subset\mathcal{E}_1(A)$, so the single curve property shown in Figure 
\ref{fig5} can be noticed also for the envelopes. In Figure \ref{fig13} we illustrate this for a complex 
random $5\times 5$ matrix  $A_4$  with real and imaginary parts of the elements between -1 and 1.
The cubic envelope $\mathcal{E}_1(A_4)$ successfully  isolates one of the eigenvalues,
marked by small boxes, but not $\mathcal{E}_2(A_4)$.
In the fourth subfigure we observe how the intersection $\mathcal{E}_1(A_4)\cap\mathcal{E}_2(A_4)$
of both regions determines a new smaller region for the spectrum.

\begin{figure}[!ht]
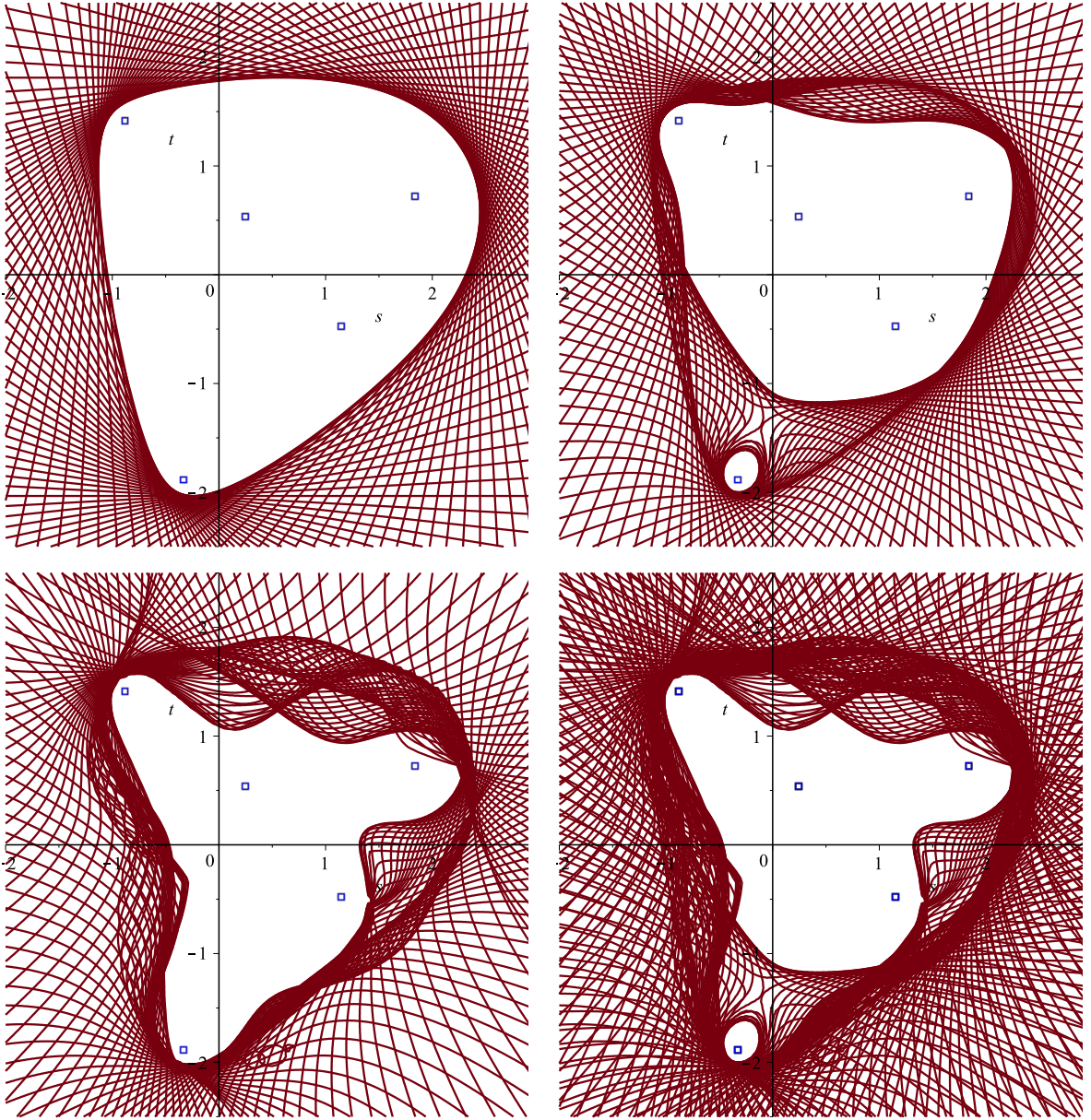

\includepdf[pages=47,pagecommand={},offset=-37mm 31mm,
trim = 25mm 120mm 60mm 20mm,clip,width=75mm]{EVFig.pdf}
\includepdf[pages=48,pagecommand={},offset=42mm 31mm,
trim = 25mm 120mm 60mm 20mm,clip,width=75mm]{EVFig.pdf}
\includepdf[pages=49,pagecommand={},offset=-37mm -51mm,
trim = 25mm 120mm 60mm 20mm,clip,width=75mm]{EVFig.pdf}
\includepdf[pages=50,pagecommand={},offset=42mm -51mm,
trim = 25mm 120mm 60mm 20mm,clip,width=75mm]{EVFig.pdf}
\vskip150mm
\caption{\label{fig13}The numerical range $F(A_4)$ (top left),  the regions $\mathcal{E}_1(A_4)$
(top right), $\mathcal{E}_2(A_4)$ (bottom left), and the intersection 
$\mathcal{E}_1(A_4)\cap\mathcal{E}_2(A_4)$ (bottom right)
 of a complex random  $5\times 5$  matrix $A_4$. }
\end{figure}

%----------------------------------------------------------------------------------------------------------------------

\newpage

\section{Open problems and discussion}

Adam, Psarrakos and Tsatsomeros \cite{AT,PT1,PT2} have proven several properties of
$\Gamma_1(A)$ and $\mathcal{E}_1(A)$ that need to be studied for $k\ge 2$. 

They proved that
when there is a closed loop of $\Gamma_1(A)$, then there is precisely one simple eigenvalue 
 inside. We have seen that we can construct matrices where $\Gamma_k(A)$ has a closed 
 loop enclosing $k$ eigenvalues, or, more generally for any $1\le j\le k$, $j$ closed loops 
 enclosing $n_1,\dots,n_j$ eigenvalues  respectively, for any positive $n_1,\dots,n_j$
 with  $j\le n_1+\dots + n_j \le k$.  It would be desirable to have a theorem giving a 
 more precise description of possible  topologies of $\Gamma_k(A)$.

The important result that $\mathcal{E}_1(A)$ is always compact is proved 
in \cite{PT2}, and it is likely that this should hold also for $\mathcal{E}_k(A)$ for $k\ge 2$.
The study of normal matrices or normal eigenvalues \cite{PT1,PT2} should also be generalized.

We have seen that $\Gamma_2(A)$ is not always more restricting for the spectrum than
$\Gamma_1(A)$ everywhere in the complex plane, and that  $\mathcal{E}_2(A)$ is not 
necessarily a subset of $\mathcal{E}_1(A)$. These facts need to be further studied
and relations between different $\Gamma_k(A)$ or different $\mathcal{E}_k(A)$ need to be explored.

The level of improvement for increasing $k$, i.e., the reduction in size of $\mathcal{E}_k(A)$
which depends on the separation of $\delta_1$, $\delta_2$,  $\delta_3, \dots$ might be
possible to quantify. For large complex random matrices with most eigenvalues within some bound,
the improvement will be smaller than for certain types of more structured matrices.

The computational complexity increases fast with increasing $k$, especially for the envelope
$\mathcal{E}_k(A)$. It might be possible to find less restrictive inequalities for the spectrum, 
but which produce curves and regions that are easier to analyze and construct numerically.

%In general, it would be desirable to analytically study more properties of $\Gamma_k(A)$
%and $\mathcal{E}_k(A)$ for $k\ge 2$.

%----------------------------------------------------------------------------------------------------------------------

%\subsection*{Acknowledgment}

%The author is grateful to 

%----------------------------------------------------------------------------------------------------------------------


\begin{thebibliography}{99}



 \bibitem{AT}
     M Adam  and M J Tsatsomeros
    "An eigenvalue inequality and spectrum localization for complex matrices"
    \textit{Electr. J. Lin. Alg.}
    \textbf{15} (2006) 239--250   

\bibitem{HJ2}
    R A Horn and C R Johnson
     \textit{Topics in Matrix Analysis}
     Cambridge University Press (Cambridge), 1991
   
 \bibitem{J}
     C R Johnson
    "Numerical determination of the field of values of a general complex matrix"
    \textit{SIAM J. Numer. Anal. }
    \textbf{15} (1978) 595--602   
    
 \bibitem{M}
     R Mathias
    "Matrices with positive definite Hermitian part: Inequalities and linear systems"
    \textit{SIAM J. Matrix Anal. Appl.}
    \textbf{13} (1992) 640--654

 \bibitem{PT1}
     P J Psarrakos  and M J Tsatsomeros
    "An envelope for the spectrum of a matrix"
    \textit{Cent. Eur. J. Math.}
    \textbf{10} (2012) 292--302       
    
 \bibitem{PT2}
     P J Psarrakos  and M J Tsatsomeros
    "On the geometry of the envelope of a matrix"
    \textit{Appl. Math. Comp.}
    \textbf{244} (2014) 132--141       

    
\end{thebibliography}
\end{document}